\documentclass{amsart}
\usepackage{amsmath,amssymb,graphicx,latexsym}
\usepackage[all]{xy}
\usepackage[bbgreekl]{mathbbol}
\newtheorem{thm}{Theorem}[section]
\newtheorem{cor}[thm]{Corollary}
\newtheorem{lem}[thm]{Lemma}
\newtheorem{prop}[thm]{Proposition}
\theoremstyle{definition}
\newtheorem{defn}[thm]{Definition}
\newtheorem{example}[thm]{Example}
\theoremstyle{remark}
\newtheorem{rem}[thm]{Remark}
\numberwithin{equation}{section}
\def\Cb{{\mathbb C}}
\def\Rb{{\mathbb R}}
\def\Zb{{\mathbb Z}}
\def\Ac{{\mathcal A}}
\def\Cc{{\mathcal C}}
\def\Fc{{\mathcal F}}
\def\Gc{{\mathcal G}}
\def\Hc{{\mathcal H}}
\def\Kc{{\mathcal K}}

\def\Nc{{\mathcal N}}

\newcommand{\mto}[1]{\stackrel{#1}{\rightarrow}}
\newcommand{\mTo}[1]{\stackrel{#1}{\longrightarrow}}

\newcommand{\To}{\longrightarrow}
\newcommand{\hf}{\mathfrak{h}} 
\newcommand{\Uf}{\mathfrak{U}}  
\newcommand{\puh}{\mathbb{P} U (\mathfrak{h})}
\newcommand{\lra}{\longrightarrow}
\newcommand{\arrows}{\rightrightarrows} 
\newcommand{\Xc}{\mathcal{X}}
\newcommand{\isom}{\simeq}
\newcommand{\Uone}{{ U}(1)}

\newcommand{\Morita}{\stackrel{Morita}{\sim}}

\DeclareMathOperator{\Hom}{Hom}

\DeclareMathOperator{\ev}{ev}   

\DeclareMathOperator{\tot}{tot} 
\DeclareMathOperator{\Aut}{Aut}
\DeclareMathOperator{\Prin}{Prin}
\DeclareMathOperator{\ad}{ad}
\DeclareMathOperator{\Br}{Br}
\DeclareMathOperator{\opp}{op}
\begin{document}

\title{A groupoid approach to noncommutative T-duality}
\author{Calder Daenzer}
\thanks{The research reported here was
supported in part by National Science Foundation grants DMS-0703718 and DMS-0611653.}
\address{970 Evans Hall, University of California at Berkeley, Berkeley, CA 94720-3840}%
\email{cdaenzer@math.berkeley.edu}

\maketitle
\begin{abstract}
Topological T-duality is a transformation taking a gerbe on a principal torus bundle to a gerbe on a principal dual-torus bundle.  We give a new geometric construction of T-dualization, which allows the duality to be extended in following two directions.  First, bundles of groups other than tori, even bundles of some nonabelian groups, can be dualized.  Second, bundles whose duals are families of noncommutative groups (in the sense of noncommutative geometry) can be treated, though in this case the base space of the bundles is best viewed as a topological stack.  Some methods developed for the construction may be of independent interest.  These are a Pontryagin type duality that interchanges commutative principal bundles with gerbes, a nonabelian Takai type duality for groupoids, and the computation of certain equivariant Brauer groups.
\end{abstract}
\tableofcontents
\newpage
\section{Introduction}
A principal torus bundle with $\Uone$-gerbe and a principal dual-torus bundle\footnote{ If a torus is written $V/\Lambda$, where $V$ is a real vector space and $\Lambda$ a full rank lattice, then its dual is the torus $\widehat{\Lambda}:=\Hom(\Lambda,U\!(1))$.}
with $\Uone$-gerbe are said to be topologically T-dual when there is an isomorphism between the twisted $K$-theory groups of the two bundles, where the ``twisting'' of the $K$-groups is determined by the gerbes on the two bundles.

The original motivation for the study of T-duality comes from theoretical physics, where it describes several phenomena and is by now a fundamental concept.  For example T-duality provides a duality between type IIa and type IIb string theory and a duality on type I string theory (see e.g. \cite{Pol}), and it provides an interpretation of a certain sector of mirror symmetry on Calabi-Yau manifolds (see \cite{SYZ}).

There have been several approaches to constructing T-dual pairs, each with their particular successes.  For example Bunke, Rumpf and Schick (\cite{BRS}) have given a description using algebraic topology methods which realizes the duality functorially.  This method is very successful for cases in which T-duals exist as commutative spaces, and has recently been extended (\cite{BSST}) to abelian groups other than tori.  In complex algebraic geometry, T-duality is effected by the Fourier-Mukai transform; in that context duals to certain toric fibrations with singular fibers (so they are not principal bundles) can be constructed (e.g. \cite{DP},\cite{BBP}).  Mathai and Rosenberg have constructed T-dual pairs using $C^*$-algebra methods, and with these methods arrived at the remarkable discovery that in certain situations one side of the duality must be a family of noncommutative tori (\cite{MR}).

In this paper we propose yet another construction of T-dual pairs, which can be thought of as a construction of the geometric duality that underlies the $C^*$-algebra duality of the Mathai-Rosenberg approach.  To validate the introduction of yet another T-duality construction, let us immediately list some of the new results which it affords.  Any of the following language which is not standard will be reviewed in the body of the paper.
\begin{itemize}
\item Duality for groups other than tori can be treated, even groups which are not abelian.  More precisely, if $N$ is a closed normal subgroup of a Lie group $G$, then the dual of any $G/N$-bundle $P\to X$ with $\Uone$-gerbe can be constructed as long as the gerbe is ``equivariant'' with respect to the translation action of $G$ on the $G/N$-bundle.  The dual is found to be an $N$-gerbe over $X$, with a $\Uone$-gerbe on it.  The precise sense in which it is a duality is given by what we call nonabelian Takai duality for groupoids, which essentially gives a way of returning from the dual side to something canonically Morita equivalent to its predual.  A twisted $K$-theory isomorphism is not necessary for there to be a nonabelian Takai duality between the two objects, though we show that there is nonetheless a $K$-isomorphism whenever $G$ is a simply connected solvable Lie group.

\item Duality can be treated for torus bundles (or more generally $G/N$ bundles, as above) whose base is a topological stack rather than a topological space.  Such a generalization is found to be crucial for the understanding of duals which are noncommutative (in the sense of noncommutative geometry).

\item We find new structure in noncommutative T-duals.  For example in the case of principal $T$-bundles, where $T=V/\Lambda\isom\Rb^n/\Zb^n$ is a torus, we find that a noncommutative T-dual is in fact a deformation of a $\Lambda$-gerbe over the base space $X$.  The $\Lambda$-gerbe will be given explicitly, as will be the 2-cocycle giving the deformation, and when we restrict the $\Lambda$-gerbe to a point $m\in X$, so that we are looking at what in the classical case would be a single dual-torus fiber, the (deformed) $\Lambda$-gerbe is presented by a groupoid with twisting 2-cocycle, whose associated twisted groupoid algebra is a noncommutative torus.  Thus the twisted $C^*$-algebra corresponding to a groupoid presentation of the deformed $\Lambda$-gerbe is a family of noncommutative tori, which matches the result of the Mathai-Rosenberg approach, but we now have an understanding of the ``global'' structure of this object, which one might say is that of a $\Lambda$-gerbe fibred in noncommutative dual tori.  Another benefit to our setup is that a cohomological classification of noncommutative duals is available, given by groupoid (or stack) cohomology.

\item  Groupoid presentations are compatible with extra geometric structure such as smooth, complex or symplectic structure.  This will allow, in particular, for the connection between topological T-duality and the complex T-duality of \cite{DP} and \cite{BBP} to be made precise.  We will begin an investigation of this and possible applications to noncommutative homological mirror symmetry in a forthcoming paper with Jonathan Block \cite{BD}.
\end{itemize}

Let us now give a brief outline of our T-dualization construction.  For the outline to be intelligible, the reader should be familiar with groupoids and gerbes or else should browse Sections \eqref{S:groupoidsAndGgroupoids}-\eqref{S:Examples} and \eqref{S:gerbes}.

Let $N$ be a closed normal subgroup of a locally compact group $G$, let $P\to X$ be a principal $G/N$-bundle over a space $X$, and suppose we are given a \v{C}ech 2-cocycle $\sigma$ on $P$ with coefficients in the sheaf of $\Uone$-valued functions.  It is a classic fact that $\sigma$ determines a $\Uone$-gerbe on $P$, and that such gerbes are classified by the \v{C}ech cohomology class of $\sigma$, written $[\sigma]\in\check{H}^2(P;\Uone)$.  So $\sigma$ represents the gerbe data.  (The case which has been studied in the past is $G\simeq\Rb^n$ and $N\simeq\Zb^n$, so that $P$ is a torus bundle.)
Given this data $(P,[\sigma])$ of a principal $G/N$-bundle $P$ with $\Uone$-gerbe, we construct a T-dual according to the following prescription:
\begin{enumerate}
\item  Choose a lift of $[\sigma]\in\check{H}^2(P;\Uone)$ to a 2-cocycle $[\tilde{\sigma}]\in\check{H}^2_G(P;\Uone)$ in $G$-equivariant \v{C}ech cohomology (see Section \eqref{S:equivariantcoho}).  If no lift exists there is no T-dual in our framework.
\item From the lift $\tilde{\sigma}$, define a new gerbe as follows.  By definition, $\tilde{\sigma}$ will be realized as a $G$-equivariant 2-cocycle in the groupoid cohomology of some groupoid presentation $\Gc(P)$ of $P$ (see Example \eqref{ex:bundlegroupoid}).  Because $\tilde{\sigma}$ is $G$-equivariant, it can be interpreted as a 2-cocycle on the crossed product groupoid $G\ltimes\Gc(P)$ for the translation action of $G$ on $\Gc(P)$ (see Example \eqref{ex:transformationgroupoid}).  Thus $\tilde{\sigma}$ determines a $\Uone$-gerbe on the \emph{groupoid} $G\ltimes\Gc(P)$.
\item  The crossed product groupoid $G\ltimes\Gc(P)$ is shown to present an $N$-gerbe over $X$, so $\tilde{\sigma}$ is interpreted as the data for a $U(1)$-gerbe on this $N$-gerbe.  This $\Uone$-gerbe on an $N$-gerbe can be viewed as the T-dual (there will be ample motivation for this).  We construct a canonical induction procedure, nonabelian Takai duality (see Section \eqref{S:nonabTakai}), that recovers the data $(P,\tilde{\sigma})$ from this $T$-dual.
\item In the special case that $N$ is abelian and $\tilde{\sigma}$ has a vanishing ``Mackey obstruction'', we construct a ``Pontryagin dual'' of the $\Uone$-gerbe over the $N$-gerbe of Step (3).  This dual object is a principal $G/N^{dual}$-bundle with $\Uone$-gerbe, where $G/N^{dual}\equiv\widehat{N}:=\Hom(N;\Uone)$ is the Pontryagin dual\footnote{For example when $N\simeq\Zb^n$, $\widehat{N}$ is the $n$-torus which is (by definition) dual to $\Rb^n/\Zb^n$.}.  Thus in this special case we arrive at a classical T-dual, which is a principal $G/N^{dual}$-bundle with $\Uone$-gerbe, and the other cases in which we cannot proceed past Step (3) are interpreted as noncommutative and nonabelian versions of classical T-duality.
\end{enumerate}
The above steps are Morita invariant in the appropriate sense and can be translated into statements about stacks.  Furthermore, they produce a unique dual object (up to Morita equivalence or isomorphism) once a lift $\tilde{\sigma}$ has been chosen.  It should be noted, however, that neither uniqueness nor existence are intrinsic feature of a T-dualization whose input data is only $(P,[\sigma])$.  In fact, the different possible T-duals are parameterized by the fiber over $[\sigma]$ of the forgetting map $\check{H}^2_G(P;\Uone)\to\check{H}^2(P;\Uone)$, which is in general neither injective nor surjective.  In some cases the forgetful map is injective.  For example this is true  1-dimensional tori, and consequently T-duals of gerbes over principle circle bundles are unique.

At the core of our construction is the concept of dualizing by taking a crossed product for a group action.  This concept was first applied by Jonathan Rosenberg and Mathai Varghese, albeit in a quite different setting than ours.  We have included an appendix which makes precise the connection between our approach and the approach presented in their paper \cite{MR}.  The role of Pontryagin duality in T-duality may have been first noticed by Arinkin and Beilinson, and some notes to this effect can be found in Arinkin's appendix in \cite{DP}, (though this is in the very different setting of complex T-duality).  The idea from Arinkin's appendix has recently been expanded upon in the topological setting in \cite{BSST}.  Our version of Pontryagin duality almost certainly coincides with these, though we arrived at it from a somewhat different perspective.
\ \newline\newline
\noindent\textbf{Acknowledgements. } I would like to thank Oren Ben-Bassat, Tony Pantev, Michael Pimsner, Jonathan Rosenberg, Jim Stasheff, and most of all Jonathan Block, for advice and helpful discussions.  I am also grateful to the Institut Henri Poincar\'e, which provided a stimulating environment for some of this research.

\section{Groupoids and G-groupoids}\label{S:groupoidsAndGgroupoids}
Let us fix notation and conventions for groupoids.  A set theoretic groupoid is a small category $\Gc$ with all arrows invertible, written as follows:
\[ \Gc:=(\Gc_1\stackrel{s,r}{\arrows}\Gc_0). \]
Here $\Gc_1$ is the set of arrows, $\Gc_0$ is the set of units (or objects), $s$ is the source map, and $r$ is the range map.  The $n$-tuples of composable arrows will be denoted $\Gc_n$.  Throughout the paper $\gamma's$ will be used to denote arrows in a groupoid unless otherwise noted.

A topological groupoid is one whose arrows $\Gc_1$ and objects $\Gc_0$ are topological spaces and for which the structure maps (source, range, multiplication, and inversion) are continuous.

A left Haar system on a groupoid (see \cite{Ren}) is, roughly speaking, a continuous family of measures on the range fibers of the groupoid that is invariant under left groupoid multiplication.  It is shown in \cite{Ren} that for every groupoid admitting a left Haar system, the source and range maps are open maps.

In this paper, \textbf{a groupoid will mean a topological groupoid whose space of arrows is locally compact Hausdorff.  Also each groupoid will be implicitly equipped with a left Haar system}.  These extra conditions are needed so that groupoid $C^*$-algebras can be defined. \textbf{Furthermore, all groupoids will be assumed second countable}, that is, the space of arrows will be assumed second countable.  Second countability of a groupoid ensures that the groupoid algebra is well behaved.  For example, second countability implies that the groupoid algebra is separable and thus well-suited for $K$-theory; second countability is invoked in \cite{Ren} when showing that every representation of a groupoid algebra comes from a representation of the groupoid \cite{Ren}; and the condition is used in \cite{MRW} when showing that Morita equivalence of groupoids implies strong Morita equivalence of the associated groupoid algebras.

On the other hand, several results presented here do not involve groupoid algebras in any way.  It will hopefully be clear in these situations that the Haar measure and second countability hypotheses, and in some cases local compactness, are unnecessary.

Topological groups and topological spaces are groupoids, and they will be assumed here to satisfy the same implicit hypotheses as groupoids.  Thus spaces and groups are always second countable, locally compact Hausdorff, and equipped with a left Haar system of measures.

A group $G$ can act on a groupoid, forming what is called a $G$-groupoid.
\begin{defn} A \textbf{(left) $G$-groupoid} is a groupoid $\Gc$ with a (continuous) left $G$-action on its space of arrows that commutes with all structure maps and whose Haar system is left $G$-invariant.
\end{defn}
\section{Modules and Morita equivalence for groupoids}\label{S:groupoidmoritaeq}
Let $\Gc$ be a groupoid.  A \textbf{left $\Gc$-module} is a space $P$ with a continuous map $P\stackrel{\varepsilon}{\to}\Gc_0$ called the \textbf{moment map} and a continuous ``action''
\[ \Gc\times_{\Gc_0}P\to P;\quad (\gamma,p)\mapsto\gamma p. \]
Here $\Gc_1\times_{\Gc_0}P:=\{\ (\gamma,p)\ |\ s\gamma=\varepsilon p\ \}$ is the fibred product and by ``action'' we mean that $\gamma_1(\gamma_2p)=(\gamma_1\gamma_2)p$.

A right module is defined similarly, and one can convert a left module $P$ to a right module $P^{\opp}$ by setting
\[ p\cdot\gamma:=\gamma^{-1}p;\quad \gamma\in\Gc\ ,\ p\in P^{\opp}.  \]
A $\Gc$-action is called \textbf{free} if $(\gamma p=\gamma' p)\Rightarrow(\gamma=\gamma')$ and is called \textbf{proper} if the map
\[ \Gc\times_{\Gc_0}P\to P\times P;\quad (\gamma,p)\mapsto(\gamma p,p)  \]
is proper.  A $\Gc$-module is called \textbf{principal} if the $\Gc$ action is both free and proper, and is called \textbf{locally trivial} when the quotient map $P\to\Gc\backslash P$ admits local sections.

Note that when $\Gc=(G\arrows *)$ is a group, a locally trivial principal $\Gc$-module $P$ is exactly a principal $G$-bundle over the quotient space $G\backslash P$.  For this reason principal modules are sometimes called principal bundles.  We are reserving the term principal bundle for something else (see Example \eqref{ex:bundlegroupoid}).

Now we come to the important notion of \emph{groupoid Morita equivalence}.
\begin{defn}\label{D:groupoidME} Two groupoids $\Gc$ and $\Hc$ are said to be \textbf{Morita equivalent} when there exists a \textbf{Morita equivalence $(\Gc$-$\Hc)$-bimodule}.  This is a space $P$ with commuting left $\Gc$-module and right $\Hc$-module structures that are both principal, and satisfying the following extra conditions.
\begin{itemize}
\item The quotient space $\Gc\backslash P$ (with its quotient topology) is homeomorphic to $\Hc_0$ in a way that identifies the right moment map $P\to\Hc_0$ with the quotient map $P\to\Gc\backslash P$.
\item The quotient space $P/\Hc$ (with its quotient topology) is homeomorphic to $\Gc_0$ in a way that identifies the left moment map $P\to\Gc_0$ with the quotient map $P\to P/\Hc$.
\end{itemize}
\end{defn}
In the literature on groupoids one finds several other ways to express Morita equivalence, but they are all equivalent (see for example \cite{BX}).

Morita equivalence bimodules give rise to equivalences of module categories.  To see this, let $E$ be a right $\Gc$-module and $P$ a Morita $(\Gc$-$\Hc)$-bimodule.  Then $\Gc$ acts on $E\times_{\Gc_0}P$ by
\[ \gamma\cdot(e,p):=(e\gamma^{-1},\gamma p) \]
and one checks that the right $\Hc$-module structure on $P$ induces one on $E*P:=\Gc\backslash(E\times_{\Gc_0}P)$.  The assignment $E\to E*P$ induces the desired equivalence of module categories.  The inverse is given by $P^{\opp}$; in fact the properties of Morita bimodules ensure an isomorphism of $(\Gc$-$\Gc$)-bimodules, $P*P^{\opp}\isom\Gc$, and this in turn induces an isomorphism $((E*P)*P^{\opp})\simeq E$.  If $E$ is principal then so is $E*P$, and if furthermore, both $E$ and $P$ are locally trivial, then so is $E*P$.

There is also a notion of \textbf{G-equivariant Morita equivalence} of $G$-groupoids.  This is given by a Morita equivalence $(\Gc$-$\Hc)$-bimodule $P$ with compatible $G$-action. The compatibility is expressed by saying that the map
\[ \Gc\times_{\Gc_0}P\times_{\Hc_0}\Hc\to P;\quad (\gamma,p,\eta)\mapsto\gamma p\eta \]
satisfies, for $g\in G$,
\begin{equation}\label{eq:GmoritaEquiv}
g(\gamma p\eta)=g(\gamma) g(p)g(\eta).
\end{equation}

\section{Some relevant examples of groupoids and Morita equivalences}\label{S:Examples}
Here are some groupoids and Morita equivalences which will be used throughout the paper.

\begin{example}\label{ex:cechgroupoid} \textbf{\v{C}ech groupoids and refinement}.  If $\Uf:=\{U_i\}_{i\in I}$ is an open cover of a topological space X then the \v{C}ech groupoid of the cover, which we denote $\Gc_{\Uf}$, is defined as follows:
\[ \Gc_{\Uf}:=(\coprod_{I\times I} U_{ij}\arrows\coprod_I U_i)\qquad
\begin{cases}
s:U_{ij}\hookrightarrow U_j \\
r:U_{ij}\hookrightarrow U_i \\
\end{cases}  \]
This groupoid is Morita equivalent to the unit groupoid $X\arrows X$.  Indeed, $\Gc_0$ is a Morita equivalence bimodule.  It is a right $\Gc$ module in the obvious way.  As for the left $(X\arrows X)$-module structure, the moment map $\Gc_0\to X$ is ``glue the cover together'' and the $X$-action is the trivial one $X\times_X\Gc_0\to\Gc_0$.

More generally, let $\Gc$ be any groupoid and suppose $\Uf:=\{U_i\}_{i\in I}$ is a locally finite cover of $\Gc_0$.  Define a new groupoid
\[ \Gc_{\Uf}:=(\coprod_{I\times I}\Gc^{ij}\arrows\coprod_I U_i)\text{ where }\Gc^{ij}:=r^{-1}U_i\cap s^{-1}U_j \]
The source and range maps are $\coprod s^{ij}$ and $\coprod r^{ij},$ where
$\begin{cases}
s^{ij}:\Gc^{ij}\hookrightarrow s^{-1}U_j\stackrel{s}{\to}U_j \\
r^{ij}:\Gc^{ij}\hookrightarrow r^{-1}U_j\stackrel{r}{\to}U_i.\\
\end{cases}$
We will call such a groupoid a \textbf{refinement} of $\Gc$ and write $\gamma^{ij}$ for $\gamma\in \Gc^{ij}$.  A groupoid is always Morita equivalent to its refinements.  A Morita equivalence $(\Gc$-$\Gc_{\Uf})$-bimodule $P$ is defined as follows:
\[ P:=\coprod_I s^{-1}U_i. \]
For a left $\Gc$-module structure on $P$, let the moment map be $r:P\to\Gc_0$ and, writing $\eta^i$ for $\eta\in s^{-1}U_i\subset P$, define the action by
\[ (\gamma,\eta^i)\mapsto (\gamma\eta)^i\qquad\gamma\in\Gc,\ \eta^i\in P. \]
For the right $\Gc_{\Uf}$-module structure, the moment map is $\eta^i\mapsto s\eta\in U_i$ and the right action is $(\eta^i,\gamma^{ij})\mapsto(\eta\gamma)^j$.

Of course a \v{C}ech groupoid is exactly a refinement of a unit groupoid.  In order to keep within our class of second countable groupoids, we restrict to countable covers of $\Gc_0$.
\end{example}
\begin{example}\label{ex:transformationgroupoid} \textbf{Crossed product groupoids.} From a $G$-groupoid $\Gc$ one can form the crossed product groupoid
\[ G\ltimes\Gc:=(G\times\Gc_1\arrows\Gc_0) \]
whose source and range maps are $s(g,\gamma):=s(g^{-1}\gamma)$ and \ $r(g,\gamma):=r\gamma$, and for which a composed pair looks like:
\[ (g,\gamma)\circ(g',g^{-1}\gamma')=(gg',\gamma\gamma'). \]

Now suppose two $G$-groupoids $\Gc$ and $\Hc$ are equivariantly Morita equivalent via a bimodule $P$ with moment maps $b_\ell:P\to\Gc_0$ and $b_r:P\to\Hc_0$.  Then $G\times P$ has the structure of a Morita $(G\ltimes\Gc)$-$(G\ltimes\Hc)$-bimodule.  The left $G\ltimes\Gc$ action is
\[ (g,\gamma)\cdot (g',p):=(gg',\gamma gp)\qquad (g,\gamma)\in G\ltimes\Gc,\ (g',p)\in G\times P \]
with moment map $(g',p)\mapsto b_\ell(p)$.  The right $G\ltimes\Hc$-module structure is
\[ (g',p)\cdot(g'',\eta):= (g'g'',pg'\eta)\qquad (g'',\eta)\in G\ltimes\Hc \]
with moment map $(g',p)\mapsto b_r(g'^{-1}p)$.
\end{example}
\begin{example}\label{ex:bundlegroupoid}\textbf{Generalized principal bundles.} Let $\Gc$ be a groupoid, $G$ a locally compact group, and $\rho:\Gc\to G$ a homomorphism of groupoids.  The generalized principal bundle associated to $\rho$ is the groupoid
\[ G\rtimes_\rho\Gc:=(G\times\Gc_1\arrows G\times\Gc_0) \]
whose source and range maps are
\[ s:(g,\gamma)\mapsto (g\rho(\gamma),s\gamma)\quad\text{and}\quad r:(g,\gamma)\mapsto(g,r\gamma) \]
and for which a composed pair looks like
\[ (g,\gamma_1)\circ(g\rho(\gamma_1),\gamma_2)=(g,\gamma_1\gamma_2). \]

The reason $G\rtimes_\rho\Gc$ is called a generalized principal bundle is that when $\Gc=\check{\Gc}\{U_i\}$ is the \v{C}ech groupoid of Example \eqref{ex:cechgroupoid},  $G\rtimes_\rho\Gc$ is $G$-equivariantly Morita equivalent to a principal bundle on $X$. Indeed, in this case $\rho$ is the same thing as a $G$-valued \v{C}ech 1-cochain on the cover, and the homomorphism property $\rho(\gamma_1\gamma_2)=\rho(\gamma_1)\rho(\gamma_2)$ translates to $\rho$ being closed. Thus $\rho$ gives transition functions for the principal $G$-bundle on $X$
\[P(\rho):= \coprod G\times U_\alpha/\sim \qquad (g,u\in U_\alpha)\sim(g\rho(\gamma),u\in U_\beta),  \]
where $\gamma=u\in U_{\alpha\beta}\subset\Gc$.  Let $\pi$ denote the bundle map $P(\rho)\to X$, then there are isomorphisms
\[ h_\alpha:G\times U_\alpha\lra \pi^{-1}U_\alpha \]
which satisfy $h_\beta^{-1}h_\alpha(g,u)=(g\rho(\gamma),u)$, and the maps $h_\alpha$ give a $G$-equivariant isomorphism of groupoids between $G\rtimes_\rho\Gc$ and the \v{C}ech groupoid $\check{\Gc}(\{\pi^{-1}U_\alpha\})$ by sending $(g,\gamma)\in G\times U_{\alpha,\beta}\subset G\rtimes_\rho\Gc$ to $h_\alpha(g,\gamma)\in\pi^{-1}U_{\alpha\beta}\subset\check{\Gc}(\{\pi^{-1}U_\alpha\})$.  Finally, $\check{\Gc}(\{\pi^{-1}U_\alpha\})$ is $G$-equivariantly Morita equivalent to the unit groupoid $P(\rho)\arrows P(\rho).$

To see the importance of keeping track of $G$-equivariance, note for instance that when $G$ is abelian $G\rtimes_\rho\Gc$ and $G\rtimes_{\rho^{-1}}\Gc$ are isomorphic (and therefore Morita equivalent), whereas these two groupoids with their natural $G$-groupoid structures are not equivariantly equivalent.
\end{example}
\begin{example}\label{ex:imprimitivity}\textbf{Isotropy subgroups.} Let $G$ be a locally compact group and $N$ a closed subgroup.  Then $G$ acts on the homogeneous space $G/N$ by left translation and one can form the crossed product groupoid $G\ltimes G/N\arrows G/N$.  There is a Morita equivalence
\[ (G\ltimes G/N\arrows G/N) \sim (N\arrows *).\]
The bimodule implementing the equivalence is $G$, with $N$ acting on the right by translation and $G\ltimes G/N$ acting on the left by $(g,ghN)\cdot h:=gh$.
\end{example}
\begin{example}\label{ex:NCgerbe}\textbf{Nonabelian groupoid extensions. }  Let $\Gc$ be a groupoid and $B\to\Gc_0$ a bundle of not necessarily abelian groups over $\Gc_0$.  Suppose we have two continuous functions
\[ \Gc_2\times_{\Gc_0}B\to B\ ;\ \ (\gamma_1,\gamma_2,p)\mapsto \sigma(\gamma_1,\gamma_2)p\ \ \text{ and } \]
\[ \Gc\times_{\Gc_0}B\to B\ ;\ \  (\gamma,p)\mapsto\tau(\gamma)(p), \]
such that $\sigma(\gamma_1,\gamma_2)$ is an element of the fiber of $B$ over $r\gamma_1$, $\tau(\gamma)$ is an isomorphism from the fiber over $s\gamma$ to the fiber over $r\gamma$, and the following equations are satisfied:
\begin{align}\label{eq:ncGerbe}
&\tau(\gamma_1)\circ\tau(\gamma_2)=\ad(\sigma(\gamma_1,\gamma_2))\circ\tau(\gamma_1\gamma_2) \\
&(\tau(\gamma_1)\circ\sigma(\gamma_2,\gamma_3))\sigma(\gamma_1,\gamma_2\gamma_3)=\sigma(\gamma_1,\gamma_2)\sigma(\gamma_1\gamma_2,\gamma_3)
\end{align}
where $\ad(p)(q):=pqp^{-1}$ for elements $p,q\in B$ that both lie in the same fiber over $\Gc_0$.  We will write $\gamma(p):=\tau(\gamma)(p)$.
The pair $(\sigma,\tau)$ can be thought of as a 2-cocycle in ``nonabelian cohomology'' with values in $B$, and when $B$ is a bundle of abelian groups, $\tau$ is simply an action and $\sigma$ a 2-cocycle as in Section \eqref{S:equivariantcoho}.

From the data $(\sigma,\tau)$ we form an \textbf{extension of $\Gc$ by $B$}, which is the groupoid
\[ B\rtimes^\sigma\Gc:=(B\times_{b,\Gc_0,r}\Gc_1\arrows \Gc_0) \]
with source, range and multiplication maps
\begin{enumerate}
\item $s(p,\gamma):=s\gamma\quad r(p,\gamma):=r\gamma$
\item $(p_1,\gamma_1)\circ(p_2,\gamma_2)=(p_1\gamma_1(p_2)\sigma(\gamma_1,\gamma_2),\gamma_1\gamma_2)$
\end{enumerate}
\end{example}
The next example combines Examples \eqref{ex:bundlegroupoid}, \eqref{ex:imprimitivity}, and \eqref{ex:NCgerbe}.
\begin{example}\label{ex:bundleimprimitivity}
Let $\rho:\Gc\to G$ be a continuous function. Define \[\delta\rho(\gamma_1,\gamma_2):=\rho(\gamma_1)\rho(\gamma_2)\rho(\gamma_1\gamma_2)^{-1}\qquad (\gamma_1,\gamma_2)\in\Gc^2.\]
Suppose $\delta\rho$ takes values in a closed normal subgroup $N$ of $G$, and write $\bar{\rho}$ for the composition $\Gc\stackrel{\rho}{\to}G\to G/N$, which is a homomorphism.  Associated to $\rho$ we construct two groupoids.
\begin{enumerate}
\item $(G\ltimes(G/N\rtimes_{\bar{\rho}}\Gc).$  This is the crossed product groupoid of $G$ acting by translation on the generalized principal bundle $G/N\rtimes_{\bar{\rho}}\Gc$.  This means that for $(g,t,\gamma)'s\in(G\ltimes(G/N\rtimes_{\bar{\rho}}\Gc)$, the source, range and multiplication maps are
\begin{enumerate}
\item $s(g,t,\gamma)=(g^{-1}t\rho(\gamma),s\gamma)$
\item $r(g,t,\gamma)=(t,r\gamma)$
\item $(g_1,t,\gamma_1)\circ(g_2,g_1^{-1}t\rho(\gamma_1),\gamma_2)=(g_1g_2,t,\gamma_1\gamma_2)$
\end{enumerate}
\item $(N\rtimes^{\delta\rho}\Gc)$.  In the notation of Example \eqref{eq:ncGerbe} this is the extension determined by the pair $(\sigma,\tau):=(\delta\rho,\ad(\rho))$ and the constant bundle $B:=\Gc_0\times N$.  Note that as an $N$-valued groupoid 2-cocycle $\delta\rho$ is not necessarily a coboundary. Explicitly, the source, range and multiplication are
\begin{enumerate}
\item $s(n,\gamma)=s\gamma$
\item $r(n,\gamma)=r\gamma$
\item $(n_1,\gamma_1)\circ(n_2,\gamma_2)=(n_1\gamma_1(n_2)\delta\rho(\gamma_1,\gamma_2),\gamma_1\gamma_2)$
\end{enumerate}
where $\gamma(n):=\rho(\gamma)n\rho(\gamma)^{-1}.$
\end{enumerate}
\begin{prop}
The two groupoids $G\ltimes(G/N\rtimes_{\bar{\rho}}\Gc)$ and $(N\rtimes^{\delta\rho}\Gc)$ of Example \eqref{ex:bundleimprimitivity} are Morita equivalent.
\end{prop}
\begin{proof}
The equivalence bimodule is $P=G\times\Gc$, endowed with the following structures.
\begin{enumerate}
\item Moment maps: $P\ni(g,\gamma)\mapsto(g\rho(\gamma)^{-1},r\gamma)\in\Hc_0$ and $P\ni(g,\gamma)\mapsto s\gamma\in\Kc_0$.
\item $\Hc$-action: $\Hc\times_{\Hc_0}P\ni((g_1,t,\gamma_1),(g_2,\gamma_2))\mapsto(g_1g_2,\gamma_1\gamma_2)\in P,$ whenever\newline  $t=g_1g_2\rho(\gamma_2)^{-1}\rho(\gamma_1)^{-1}\in G/N$.
\item $\Kc$-action: $P\times_{\Kc_0}\Kc\ni ((g,\gamma_1),(n,\gamma_2))\mapsto(gn\rho(\gamma_2),\gamma_1\gamma_2)\in P$.
\end{enumerate}
Direct checks show that these definitions make $P$ a Morita equivalence bimodule.
\end{proof}
\end{example}
\noindent\textbf{Summary of notation. } For convenience, let us summarize the notation that has been developed in these examples.
\begin{itemize}
\item $(G\ltimes \Gc)$ denotes a crossed product groupoid.  It is in some sense a quotient of $\Gc$ by $G$.
\item $(G\rtimes_\rho\Gc)$ denotes a principal bundle over $\Gc$.
\item $(G\rtimes^\sigma\Gc)$ denotes an extension of $\Gc$ by $G$.  We will see that this corresponds to a presentation of a $G$-gerbe over $\Gc$.
\end{itemize}

\section{Groupoid algebras, K-theory, and strong Morita equivalence}\label{ss:groupoidalgebra}
 Let $\Gc$ be a groupoid.  The continuous compactly supported functions $\Gc_1\to\Cb$ form an associative algebra, denoted $C_c(\Gc)$, for the following multiplication called \textbf{groupoid convolution}.
\begin{equation}
a*b(\gamma):=\int_{\gamma_1\gamma_2=\gamma}a(\gamma_1)b(\gamma_2)
\end{equation}
for $a,b\in C_c(\Gc)$ and $\gamma's\in\Gc$.  Integration is with respect to the fixed left Haar system of measures.
This algebra has an involution,
\[ a\mapsto a^*(\gamma)=\overline{a(\gamma^{-1})} \]
(the overline denotes complex conjugation), and can be completed in a canonical way to a $C^*$-algebra (see \cite{Ren}) which we simply refer to as the \textbf{groupoid algebra} and denote $C^*(\Gc)$ or $C^*(\Gc_1\arrows\Gc_0)$.

The groupoid algebra is a common generalization of the continuous functions on a topological space, to which this reduces when $\Gc$ is the unit groupoid, and of the convolution $C^*$-algebra of a locally compact group, to which this reduces when the unit space is a point.  Indeed, by definition of the groupoid algebra we have $C^*(X\arrows X)=C(X)$ and $C^*(G\arrows *)=C^*(G)$ when $X$ is a locally compact Hausdorff space and $G$ is a locally compact Hausdorff group.

As is probably common, we will define the \textbf{K-theory of $\Gc$}, denoted $K(\Gc)$, to be the $C^*$-algebra K-theory of its groupoid algebra.  Here are the facts we need about groupoid algebras and $K$-theory:
\begin{prop}\label{P:moritaFacts}\ \newline
\begin{enumerate}
\item \cite{MRW} A Morita equivalence of groupoids gives rise to a (strong) Morita equivalence of the associated groupoid algebras.
\item  A Morita equivalence between $G$-groupoids $\Gc$ and $\Hc$ gives rise to a Morita equivalence between the crossed product $C^*$-algebras $G\!\ltimes\! C^*(\Gc)$ and $G\!\ltimes\! C^*(\Hc)$.
\item Groupoid $K$-theory is invariant under Morita equivalence.
\end{enumerate}
\end{prop}
\begin{proof} The first statement is the main theorem of \cite{MRW}.  The second statement follows from the first after noting that the definitions of $G\!\ltimes\! C^*(\Gc)$ and $C^*(G\ltimes\Gc)$ coincide and that $G\ltimes\Gc$ is Morita equivalent to $G\ltimes\Hc$ (see Example \eqref{ex:transformationgroupoid}).  The last statement now follows from the Morita invariance of $C^*$-algebra $K$-theory.
\end{proof}
Let $\Gc$ and $\Hc$ be Morita equivalent groupoids.  It is useful to know that in \cite{MRW} a $C^*(\Gc)$-$C^*(\Hc)$-bimodule is constructed directly from a $\Gc$-$\Hc$-Morita equivalence bimodule $P$.  The $C^*$-algebra bimodule is a completion of $C_c(P)$ and has the actions induced from the translation actions of $\Gc$ and $\Hc$ on $P$.  We present a generalization of this in Lemma \eqref{L:appendixLemma1}.

\section{Equivariant groupoid cohomology}\label{S:equivariantcoho}
In this section we define equivariant groupoid cohomology for $G$-groupoids.  Equivariant 2-cocycles will give rise to what we call \emph{equivariant gerbes}.

Let $\Hc$ be a groupoid and $B\stackrel{b}{\to}\Hc_0$ a left $\Hc$-module each of whose fibers over $\Hc_0$ is an abelian group, that is a (not necesssarily locally trivial) \textbf{bundle of groups} over $\Hc_0$.  Then one defines the \textbf{groupoid cohomology with $B$ coefficients}, denoted $H^*(\Hc;B)$, as the cohomology of the complex $(C^\bullet(\Hc;B),\delta)$, where
\[ C^k(\Hc;B):=\{\text{ continuous maps } f:\Hc_k\to B\ |\ b(f(h_1,\dots,h_n))=rh_1\ \} \]
and for $f\in C^k(\Hc;B),$
\[ \delta f(h_1,\dots,h_{k+1}):=h_1\cdot f(h_2,\dots,h_{k+1})
+\sum_{i=1\dots k}(-1)^if(h_1,\dots,h_ih_{i+1},\dots,h_{k+1})\]
\[
+(-1)^{k+1}f(h_1,\dots,h_{k+1}).
\]
As is common, we tacitly restrict to the quasi-isomorphic subcomplex
\[ \{f\in C^k\ |\ f(h_1,\dots,h_k)=0\text{ if some }h_i\text{ is a unit }\},\] except for 0-cochains which have no such restriction.
When the $\Hc$-module is $B=\Hc_0\times A$, where $A$ is an abelian group, we write $A$ for the cohomology coefficients.

When $\Hc$ is the \v{C}ech groupoid of a locally finite cover of a topological space $X$ and $B$ is the \'etale space of a sheaf of abelian groups on $X$, then $C^\bullet$ is identical to the \v{C}ech complex of the cover with coefficients in the sheaf of sections of $B$, so this recovers \v Cech cohomology of the given cover.  On the other hand, when $\Hc$ is a group this recovers continuous group cohomology.
\newline

If $\Hc$ is a $G$-groupoid then $C^\bullet(\Hc;B)$ becomes a complex of left $G$-modules by
\[ g\cdot f(h_1,\dots,h_n):=f(g^{-1}h_1,\dots,g^{-1}h_n)\qquad f\in C^k(\Hc;B), \]
and one can form the double complex
\[ K^{p,q}=(C^p(G;C^q(\Hc;B)),d,\delta), \]
where $d$ denotes the groupoid cohomology differential for $G\arrows*$.
The \textbf{$G$-equivariant cohomology} of $\Hc$ with values in $B$, denoted $H_G^*(\Hc;B)$, is the cohomology of the total complex
\[ \tot(K)^n:=(\oplus_{p+q=n}K^{p,q},D=d+(-1)^p\delta). \]

As one would hope, there is a chain map from the complex $\tot(K)$ computing equivariant cohomology to the chain complex associated to the crossed product groupoid:
\begin{prop} The map
\begin{equation}\label{Eq:equicohoCrossedCoho}
F:\tot K^\bullet\To C^\bullet(G\ltimes\Hc;B), \end{equation}
defined to be the sum of the maps
\[ C^p(G;C^q(\Hc;B))\mTo{f_{pq}}C^{p+q}(G\ltimes\Hc;B) \]
\[f_{pq}(c)((g_1,\gamma_1),(g_2,g_1^{-1}\gamma_2),\dots ,(g_{p+q},(g_1g_2\dots g_{p+q-1})^{-1}\gamma_{p+q}))\]
\[:=c(g_1,\dots,g_p,\gamma_{p+1},\dots,\gamma_{p+q})\]
for $c\in C^p(G;C^q(\Hc;B))$, $g's\in G$, and $(\gamma_1,\dots,\gamma_{p+q})\in\Hc_{p+q}$, is a morphism of chain complexes.
\end{prop}
\begin{proof} This is a direct check. \end{proof}
It seems likely that this is a quasi-isomorphism, but we have not proved it.  In Section \eqref{S:cohoFacts} it is shown that $H(G\ltimes\Hc;B)$ is always a summand of $H_G(\Hc;B)$, which is enough for the present purposes.
\newline

\begin{rem} These cohomology groups are not Morita invariant.  For example, different covers can have different \v{C}ech cohomology.  One can form a Morita invariant cohomology (that is, stack cohomology); it is the derived functors of $B\mapsto \Hom_\Hc(\Hc_0,B)=\Gamma(\Hc_0,B)^\Hc$, which homological algebra tells us can be computed by using a resolution of $B$ by injective $\Hc$-modules\footnote{ There are enough injective $\Hc$-modules for \'etale groupoids, but we do not know if this is true for general groupoids.}.  However, the cocycles obtained via injective resolutions are often not useful for describing geometric objects such as bundles and groupoid extensions, so we will stick with the groupoid cohomology as defined above (which is the approximation to these derived functors obtained by resolving $\Hc_0$ by $\Hc_\bullet$ and taking the cohomology of $\Hom_\Hc(\Hc_\bullet;B)$).  In Section \eqref{S:twistedMoritaEquivalence} we will show how to compare the respective groupoid cohomology groups of two groupoids which are Morita equivalent.
\end{rem}
We will also encounter cocycles with values in a bundle of nonabelian groups $B$, defined in degrees $n=0,1,2$.  The spaces of cochains are the same and a $0$-cocycle is also the same as in the abelian setting.  In degree one we say $\rho\in C^1(\Hc;B)$ is closed when $\delta\rho(\gamma_1,\gamma_2):=\rho(h_1)h_1\cdot\rho(h_2)\rho(h_1h_2)^{-1}=1$ and $\rho$ and $\rho'$ are cohomologous when $\rho'(h)=h\cdot\alpha(sh)\rho(h)\alpha^{-1}(rh)$.
A nonabelian 2-cocycle is a pair $(\sigma,\tau)$ as in Example \eqref{ex:NCgerbe}.

\section{Gerbes and twisted groupoids}\label{S:gerbes}
In this section we describe various constructions that can be made with 2-cocycles and, in particular, explain our slightly non-standard use of the term \emph{gerbe}.  We also describe the construction of equivariant gerbes from equivariant cohomology.

Given a 2-cocycle $\sigma\in Z^2(\Gc;N)$, where $N$ is an abelian group, we can form an extension of $\Gc$ by $N$:
\[ N\rtimes^\sigma\Gc:=(N\times\Gc_1\arrows\Gc_0), \] with multiplication \[(n_1,\gamma_1)\circ(n_2,\gamma_2):=(n_1n_2\sigma(\gamma_1,\gamma_2),\gamma_1\gamma_2).\]
More generally, if $B$ is a bundle of not necessarily abelian groups, and $(\sigma,\tau)$ a $B$-valued nonabelian 2-cocycle, then we can form the groupoid extension $B\rtimes^\sigma\Gc$ that was described in Example \eqref{ex:NCgerbe}.  We will call such an extension a \textbf{$B$-gerbe}, or an \textbf{$N$-gerbe} if $B=\Gc_0\times N$ is a constant bundle of (not necessarily abelian) groups.

The term \emph{gerbe} comes from Giraud's stack theoretic interpretation of degree two nonabelian cohomology (\cite{Gir}).  In the following few paragraphs (everything up to Definition \eqref{D:twistedGroupoid}) we will outline the stack theoretic terminology leading to Giraud's gerbes.  The point of the outline is only to clear up terminology, and can be skipped.  A nice reference for topological stacks is \cite{Met}.

Let $\Cc$ be any category.  A topological stack is a functor $F:\Cc\to Top$ satisfying a certain list of axioms.  A morphism of stacks from $(F:\Cc\to Top)$ to $(F':\Cc'\to Top)$ is a functor $\alpha:\Cc\to\Cc'$ (satisfying a couple of axioms) such that $F=F'\circ\alpha$.  Such a morphism is an equivalence of stacks when $\alpha$ is an equivalence of categories.

Given a groupoid $\Gc$, define $\Prin_\Gc$ to be the category whose objects are locally trivial right principal $\Gc$-modules and whose homs are the continuous $\Gc$-equivariant maps.  This category has a natural functor to $Top$ which sends a principal module $P$ to the quotient $P/\Gc$, and in fact satisfies the axioms for a stack.  This stack (which we denote $\Prin_\Gc$) is called the stack associated to $\Gc$.  A stack which is equivalent to $\Prin_\Gc$ is called presentable and $\Gc$ is called a \textbf{presentation} of the stack.

The discussion following Definition \eqref{D:groupoidME} shows that a locally trivial right principal $\Gc$-$\Hc$-bimodule $P$ induces a functor $*P:\Prin_\Gc\to\Prin_\Hc$.  It is in fact a morphism of stacks, and is an equivalence of stacks when $P$ is also left principal (that is when $P$ is a Morita equivalence bimodule).  Conversely, if there is an equivalence of stacks $\Prin_\Gc\to\Prin_\Hc$, then $\Gc$ and $\Hc$ are Morita equivalent.  Thus any statement about groupoids which is Morita invariant is naturally a statement about presentable stacks.  We will only work with presentable stacks in this paper, and when stacks are mentioned at all, it will only be as motivation for making Morita invariant constructions.

According to Giraud \cite{Gir}, a gerbe over a stack $\Cc'$ is a stack $\Cc$ equipped with a morphism of stacks $\alpha:\Cc\to\Cc'$ satisfying a couple of axioms.  Now, the extension $B\rtimes^\sigma\Gc$ has its natural quotient map to $\Gc$ (and this quotient map is a functor), and this determines a morphism of stacks $\Prin_{(B\rtimes^\sigma\Gc)}\to\Prin_\Gc$ which in fact makes $\Prin_{(B\rtimes^\sigma\Gc)}$ into what Giraud called a \emph{$B$-gerbe} over the stack $\Prin_\Gc$ (see also \cite{Met} definition 84).  When $\Gc$ is Morita equivalent to a space $X$, one usually calls this a gerbe over $X$.  Thus we call the groupoid $B\rtimes^\sigma\Gc$ a $B$-gerbe, although it is actually a \emph{presentation} of a $B$-gerbe.  Hopefully this will not cause much confusion.

\begin{rem} The gerbes described so far have the property that the quotient map $(N\rtimes^\sigma\Gc)_1\to\Gc_1$ admits a continuous section.  In other words, since $N$ acts on an $N$-gerbe, $N\rtimes^\sigma\Gc_1$ is a trivial principal $N$-bundle (or for $B$-gerbes, a trivial $B$-torsor).  The obstruction to all gerbes being of this form is the degree one sheaf cohomology of the space $\Gc_1$ with coefficients in the sheaf of $N$-valued functions, $H^1_{Sheaf}(\Gc_1;N)$.  Since many groupoids admit a refinement for which this obstruction vanishes (in particular \'Cech groupoids do), there are plenty of situations in which one may assume the gerbe admits the above 2-cocycle description.  Nonetheless, we will encounter gerbes which are not trivial bundles, such as the ones in Example \eqref{Ex:bestgroupoid1}.
\end{rem}

Closely related to gerbes is the following notion.
\begin{defn}\label{D:twistedGroupoid} Let $\Gc$ be a groupoid and let $B\to\Gc_0$ be a bundle of groups.  A \textbf{B-twisted groupoid} is a pair $(\Gc,(\sigma,\tau))$, where $(\sigma,\tau)$ is a $B$-valued (nonabelian) 2-cocycle over $\Gc$ as in Example \eqref{ex:NCgerbe}.  When $\tau$ is understood to be trivial, we simply write $(\Gc,\sigma)$, and when $B$ is the constant bundle $\Gc_0\times\Uone$ we simply call the pair a \textbf{twisted groupoid}.
\end{defn}
In fact a $B$-twisted groupoid contains the exact same data as a $B$-gerbe.  However, we will encounter a type of duality which takes twisted groupoids to $\Uone$-gerbes and does not extend to a ``gerbe-gerbe'' duality.  Thus it is necessary to have both descriptions at hand.

We would like to make $C^*$-algebras out of twisted groupoids in order to define twisted K-theory.  Here is the definition.
\begin{defn}\cite{Ren}  Given a twisted groupoid $(\Gc,\sigma\in Z^2(\Gc;\Uone))$, the associated \textbf{twisted groupoid algebra}, denoted $C^*(\Gc,\sigma)$, is the $C^*$-algebra completion of the compactly supported functions on $\Gc_1$, with \textbf{$\sigma$-twisted multiplication}
\[ a*b(\gamma):=\int_{\gamma_1\gamma_2=\gamma}a(\gamma_1)b(\gamma_2)\sigma(\gamma_1,\gamma_2);\quad a,b\in C_c(\Gc_1) \]
and involution $a\mapsto a^*(\gamma):=\overline{a(\gamma^{-1})\sigma(\gamma,\gamma^{-1})}$.  Here functions are $\Cb$-valued and the overline denotes complex conjugation.  Of course a groupoid algebra is exactly a twisted groupoid algebra for $\sigma=1$.
\end{defn}
\begin{defn} The \textbf{twisted K-theory} of a twisted groupoid $(\Gc,\sigma)$ is the K-theory of $C^*(\Gc,\sigma)$.
\end{defn}

Now suppose $G$ is a locally compact group and $\Hc$ is a $G$-groupoid.  By definition, a $\Uone$-valued 2-cocycle in $G$-equivariant cohomology is of the form:
\begin{equation}\label{eq:twococycle}
(\sigma,\lambda,\beta)\in C^0(G;Z^2(\Hc;\Uone))\times C^1(G;C^1(\Hc;\Uone))\times Z^2(G;C^0(\Hc;\Uone)).\\ \end{equation}
and it satisfies the cocycle condition:
\[ D(\sigma,\lambda,\beta)=(\delta\sigma,\ \delta\lambda^{-1}d\sigma,\ \delta\beta d\lambda,\ d\beta)=(1,1,1,1). \]

The first component, $\sigma$, of the triple determines a twisted groupoid algebra $C^*(\Hc;\sigma)$.  Now the translation action of $G$ on $C^*(\Hc)$,
\[ g\cdot a(\gamma):=a(g^{-1}\gamma);\qquad g\in G,\ h\in\Hc,\ a\in C^*(\Hc;\sigma) \]
is not an action on $C^*(\Hc;\sigma)$ because
\[ g\cdot(a*_\sigma b)\neq(g\cdot a)*_\sigma(g\cdot b)\text{ for }a,b\in C^*(\Hc;\sigma),\ g\in G.\]
The second and third components are ``correction terms'' that allow $G$ to act on the twisted groupoid algebra.  Indeed, define a map
\[ \alpha:G\to \Aut(C^*(\Hc;\sigma))\quad \alpha_g(a)(h):=\lambda(g,h)g\cdot a(h).\]
Then we have
\[ \{\ \alpha_g(a*_\sigma b)=\alpha_g(a)*_\sigma\alpha_g(b)\}\Longleftrightarrow \{\ d\sigma=\delta\lambda\},\]
so $\alpha$ does land in the automorphisms $C^*(\Hc;\sigma)$.  However, this is still not a group homomorphism since in general $\alpha_{g_1}\circ\alpha_{g_2}\neq\alpha_{g_1g_2}$. If we attempted to construct a crossed product algebra $G\ltimes_\alpha C^*(\Hc)$ it would not be associative.  The failure of $\alpha$ to be homomorphic is corrected by the third component, $\beta$.  An interpretation of $\beta$ is that it determines a family over $\Hc_0$ of deformations of $G$ as a noncommutative space from which $\alpha$ \emph{is} in some sense a homomorphism.  We encode this ``noncommutative G-action'' in the following twisted crossed product algebra:
\[ G\ltimes_{\lambda,\beta}C^*(\Hc;\sigma), \]
which is the algebra with multiplication
\[
a*b(g,h):=\int_{\stackrel{h_1h_2=h}{g_1g_2=g}}a(g_1,h_1)b(g_2,g_1^{-1}h_2)\chi((g_1,h_1),(g_2,g_1^{-1}h_2))
\]
where
\[ \chi((g_1,h_1),(g_2,g_1^{-1}h_2)):=\sigma(h_1,h_2)\lambda(g_1,h_2)\beta(g_1,g_2,sh_2). \]
\begin{lem} Let $G\ltimes\Hc$ be the crossed product groupoid associated to the $G$ action on $\Hc$ (as in Example \eqref{ex:transformationgroupoid}).  Then $\chi\in Z^2(G\ltimes\Hc;\Uone)$.  Consequently the multiplication on
\[ G\ltimes_{\lambda,\beta}C^*(\Hc;\sigma)\equiv C^*(G\ltimes\Hc,\chi) \]
is associative.
\end{lem}
\begin{proof} $\chi$ is the image of the cocycle $(\sigma,\lambda,\beta)$ under the chain map of Equation \eqref{Eq:equicohoCrossedCoho}, thus it is a cocycle. \end{proof}
Now it is clear how to interpret the data $(\sigma,\lambda,\beta)$ at groupoid level: it is the data needed to extend the twisted groupoid $(\Hc,\sigma)$ to a twisted crossed product groupoid $(G\ltimes\Hc,\chi)$.  The meaning of ``extend'' in this context is that $(\Hc,\sigma)$ is a sub-(twisted groupoid):
\[ (\Hc,\sigma)\simeq(\{\ 1\}\ltimes\Hc,\sigma)\subset(G\ltimes\Hc,\chi), \]
which follows from the fact that $\chi|_{\Hc}\equiv\sigma$.
Clearly in the groupoid interpretation the $\Uone$ coefficients can be replaced by an arbitrary system of coefficients.
\begin{defn} The pair $(\Hc,(\sigma,\lambda,\beta))$, where $\Hc$ is a $G$-groupoid and $(\sigma,\lambda,\beta)\in Z^2_G(\Hc,\Uone)$ will be called a \textbf{twisted $G$-groupoid}.
\end{defn}

\section{Pontryagin duality for generalized principal bundles}
In this section we introduce an extension of Pontryagin duality, which has been a duality on the category of abelian locally compact groups, to a correspondence of the form
\[ \{\text{ generalized principal $G$-bundles }\}\longleftrightarrow
\{\text{ twisted }\widehat{G}\text{-gerbes }\} \]
for any abelian locally compact group $G$.  In fact we extend this to a duality between a $\Uone$-gerbe on a principal $G$-bundle (though only certain types of $\Uone$-gerbes are allowed) and a $\Uone$-gerbe on a $\widehat{G}$-gerbe.
By construction, there will be a Fourier type isomorphism between the twisted groupoid algebras of a Pontryagin dual pair; consequently any invariant constructed from twisted groupoid algebras ($K$-theory for example) will be unaffected by Pontryagin duality.

This duality might be of independent interest, especially because it is not a Morita equivalence and thus induces a nontrivial duality at stack level.  The Pontryagin dual of a $\Uone$-gerbe on a principal torus bundle will play a crucial role in the understanding of T-duality.

Let us fix the following notation for the remainder of this section:
$G$ denotes an abelian locally compact group, $\widehat{G}=\Hom(G,\Uone)$ denotes its Pontryagin dual group, and for elements of $G$ and $\widehat{G}$ we use $g's$ and $\phi's$, respectively, and write evaluation as a pairing $\langle\phi,g\rangle:=\phi(g)$.  As usual $\gamma's$ are elements of a groupoid $\Gc$.

According to our groupoid notation, $(G\arrows G)$ denotes the group $G$ thought of as a topological space while $(G\arrows *)$ denotes group thought of as a group.  Thus by definition, the groupoid algebras $C^*(G\arrows G)$ and $C^*(G\arrows *)$ are functions on $G$ with pointwise multiplication in the first case and convolution multiplication in the second.  Keeping this in mind, we Fourier transform can be interpreted as an isomorphism of groupoid algebras:
\[ \Fc: C^*(G\arrows G)\longrightarrow C^*(\widehat{G}\arrows *) \]
\[ a\mapsto \Fc(a)(\phi):=\int_{g\in G} a(g)\phi(g^{-1}). \]
We use Plancherel measure so that the inverse transform is given by
\[ \Fc^{-1}(\hat{a})(g):=\int_{\phi\in\widehat{G}}\hat{a}(\phi)\phi(g)\qquad\hat{a}\in C^*(\widehat{G}\arrows *).\]
The group $G$ acts on $C^*(G\arrows G)$ by translation
\[  g_1\cdot a(g):=a(gg_1)\qquad a\in C^*(G\arrows G)\]
and the dual group acts by ``dual translation'' on $C^*(G\arrows G)$ by
\[ \phi\star a(g):=\langle\phi,g\rangle a(g). \]
Under Fourier transform translation and dual translation are interchanged:
\begin{equation}\label{e:translation}
\Fc(g\cdot a)(\phi)=\langle\phi,g\rangle\Fc(a)(\phi)=:g\star\Fc(a)(\phi),
\end{equation}
\begin{equation}\label{e:dualtralation}
\Fc(\phi\star a)(\psi)=\Fc(a)(\phi^{-1}\psi)=:\phi^{-1}\cdot\Fc(a)(\psi),\text{ for }\phi,\psi\in\widehat{G}.
\end{equation}
Let us quickly check the first one:
\begin{align*}
\Fc(g\cdot a)(\phi)
&=\int_{g_1}a(g_1g)\phi(g_1^{-1})\\
&=\int_{g'} a(g')\phi((g'g^{-1})^{-1})\\
&=\phi(g)\int_{g'} a(g')\phi(g')=g\star\Fc(a)(\phi).
\end{align*}
With those basic rules of Fourier transform in mind, we are ready to prove:
\begin{defn}\label{D:pontDualityData} Let $G$ be a locally compact abelian group and $\Gc$ a groupoid.  The following data:
\[ \rho\in Z^1(\Gc;G),\ f\in Z^2(\Gc;\widehat{G}),\text{ and }\nu\in C^2(\Gc;\Uone), \]
satisfying $\delta\nu(\gamma_1,\gamma_2,\gamma_3)=\langle f(\gamma_1,\gamma_2),\rho(\gamma_3)^{-1}\rangle$ will be called \textbf{Pontryagin duality data}.
\end{defn}
Given Pontryagin duality data $(\rho,f,\nu)$, the following formulas define twisted groupoids:
\begin{enumerate}
\item The generalized principle bundle $(G\rtimes_\rho\Gc)$
with twisting 2-cocycle:
\[ \sigma^{\nu f}((g,\gamma_1),(g\rho(\gamma_1),\gamma_2)):=
        \nu(\gamma_1,\gamma_2)\langle f(\gamma_1,\gamma_2),g\rangle\in Z^2(G\rtimes_\rho\Gc;\Uone). \]
\item The $\widehat{G}$-gerbe $(\widehat{G}\rtimes^f\Gc)$
with twisting 2-cocycle:
\[ \tau^{\rho\nu}((\phi_1,\gamma_1),(\phi_2,\gamma_2)):=
        \nu(\gamma_1,\gamma_2)\langle\phi_2,\rho(\gamma_1)\rangle\in Z^2(\widehat{G}\rtimes^f\Gc;\Uone). \]
\end{enumerate}
To verify this, simply check that the twistings are indeed 2-cocycles. so that the pairs $(G\rtimes_\rho\Gc, \sigma^{\nu f})$ and $(\widehat{G}\rtimes^f\Gc;\tau^{\rho\nu})$ are actually twisted groupoids.
\begin{thm}[\textbf{ Pontryagin duality for groupoids }]\label{T:pontryagin}\
Let $(G\rtimes_\rho\Gc, \sigma^{\nu f})$ and $(\widehat{G}\rtimes^f\Gc;\tau^{\rho\nu})$ be twisted groupoids constructed from Pontryagin duality data $(\rho,f,\nu)$ as in Definition \eqref{D:pontDualityData}.  Then there is a Fourier type isomorphism between the associated twisted groupoid algebras:
\[ C^*(G\rtimes_\rho\Gc;\sigma^{\nu f})\stackrel{\Fc}{\longrightarrow}
                             C^*(\widehat{G}\rtimes^f\Gc;\tau^{\rho\nu})\]
\[ a\mapsto \Fc(a)(\phi,\gamma):=\int_{g\in G} a(g,\gamma)\phi(g^{-1}),\quad \gamma\in\Gc \]
Also, the natural translation action of $G$ on $C^*(G\rtimes_\rho\Gc;\sigma)$ is taken to the dual translation analogous to Equation \eqref{e:translation}:
\[ \Fc(g\cdot a)(\phi,\gamma)=\langle\phi,g\rangle\Fc(a)(\phi,\gamma)
=:g\star\Fc(a)(\phi,\gamma). \] Note, however, that $G$ only acts by vector space automorphisms (as opposed to algebra automorphisms) unless $f=1$.
\end{thm}
\begin{proof}

The fact that $\Fc$ is an isomorphism of Banach spaces follows because ``fibrewise'' this is classical Fourier transform.  Thus verifying that $\Fc$ is a $C^*$-isomorphism is a matter of seeing that $\Fc$ takes the multiplication on the first algebra into the multiplication on the second.  Let us check.

Set $f_{1,2}=f(\gamma_1,\gamma_2)\in\widehat{G}$, $\nu_{1,2}=\nu(\gamma_1,\gamma_2)\in\Uone$, and $\rho_i=\rho(\gamma_i)\in G$.  The multiplication on $C^*(G\rtimes_\rho\Gc;\sigma^{\nu f})$ is by definition
\begin{align*}
a*b(g,\gamma):
&=\int_{\gamma_1\gamma_2=\gamma} a(g,\gamma_1)b(g\rho_1,\gamma_2)
\nu_{1,2}\langle f_{1,2},g\rangle\\
&=\int_{\gamma_1\gamma_2=\gamma} a(g,\gamma_1)f_{1,2}\star(\rho_1\cdot b)(g,\gamma_2)\nu_{1,2}.
\end{align*}
Pointwise multiplication on $G$ is transformed to convolution on $\widehat{G}$, which we denote by $\hat{*}$.  Group translation and dual group translation behave under the transform according to Equations \eqref{e:translation} and \eqref{e:dualtralation}.  Using these rules, we have
\begin{align*}
\Fc(a*b)(\phi,\gamma)
&=\int_{\stackrel{\phi_1\phi_2=\phi}{\gamma_1\gamma_2=\gamma}}\Fc(a)(\phi_1,\gamma_1)\hat{*}
     \Fc(f_{1,2}\star(\rho_1\cdot b))(\phi_2,\gamma_2)\nu_{1,2}\\
&=\int_{\stackrel{\phi_1\phi_2=\phi}{\gamma_1\gamma_2=\gamma}}\Fc(a)(\phi_1,\gamma_1)
  \Fc(b)(\phi_2f_{1,2}^{-1},\gamma_2)\langle\phi_2f_{1,2}^{-1},\rho_1\rangle\nu_{1,2}\\
&=\int_{\stackrel{\phi_1\phi_2'f_{1,2}=\phi}{\gamma_1\gamma_2=\gamma}}\Fc(a)(\phi_1,\gamma_1)
  \Fc(b)(\phi_2',\gamma_2)\langle\phi_2',\rho_1\rangle\nu_{1,2} ,\\
\end{align*}
and this last line is exactly the multiplication on $C^*(\widehat{G}\rtimes^f\Gc;\tau^{\nu\rho})$.  The statement about $G$-actions is proved by the same computation as for Equation \eqref{e:translation}.
\end{proof}
\begin{defn} A pair of a twisted generalized principal $G$-bundle and twisted $\widehat{G}$-gerbe as in Theorem \eqref{T:pontryagin} are said to be \textbf{Pontryagin dual}.
\end{defn}

This Pontryagin duality is independent of choice of cocycles $\rho$ and $f$ within their cohomology classes, and furthermore, if we alter $\nu$ by a closed 2-cochain (recall $\nu$ itself is not closed) we obtain a new Pontryagin dual pair.  Here is the precise statement of these facts; the proof is a simple computation.
\begin{prop}\label{P:PontCohoInv} Suppose $(G\rtimes_\rho\Gc;\sigma^{\nu f})$ and $(\widehat{G}\rtimes^f\Gc;\tau^{\rho\nu})$ are Pontryagin dual and we are given $\alpha\in C^1(\Gc;\widehat{G})$ and $\beta\in C^0(\Gc;G)$ and $c\in Z^2(\Gc;\Uone)$.  Define
\[ \rho':=\rho\delta\beta\ \text{ and }\ f':=f\delta\alpha. \]
Then $(G\rtimes_{\rho'}\Gc;\sigma^{\nu' f'})$ and $(\widehat{G}\rtimes^{f'}\Gc;\tau^{\rho'\nu'})$ are Pontryagin dual as well, where
\[ \nu'_{1,2}:=c_{1,2}\nu_{1,2}\langle f_{1,2},\beta(s\gamma_2)\rangle^{-1}\langle\alpha_1,\rho'_2\rangle^{-1}. \]
\end{prop}
\begin{flushright}$\square$\end{flushright}

Since Pontryagin duality is defined in terms of a Fourier isomorphism, we have the following obvious Morita invariance property:
\begin{prop}  Suppose we are given two sets of Pontryagin duality data $(\Gc,G,\rho,\nu,f)$ and $(\Gc',G,\rho',\nu',f')$. Then
\begin{enumerate}
\item $C^*(G\rtimes_\rho\Gc;\sigma^{\nu f})$ is Morita equivalent to $C^*(G\rtimes_{\rho'}\Gc';\sigma^{\nu' f'})$ if and only if $C^*(\widehat{G}\rtimes^f\Gc;\sigma^{\rho\nu})$ is Morita equivalent to $C^*(\widehat{G}\rtimes^{f'}\Gc';\sigma^{\rho'\nu'})$.
\item In particular, if $G\rtimes_\rho\Gc$ is $G$-equivariantly Morita equivalent to $G\rtimes_{\rho'}\Gc'$ and the twisted groupoid $(G\rtimes_\rho\Gc,\sigma^{\nu f})$ is Morita equivalent to $(G\rtimes_{\rho'}\Gc',\sigma^{\nu' f'})$ (see Theorem \eqref{p:twistedGroupoidMoritaEq} for Morita equivalence of twisted groupoids), then all $C^*$-algebras in $(1)$ are Morita equivalent.  The same is true if the $\widehat{G}$-twisted groupoids $(\Gc,f)$ is Morita equivalent to $(\Gc',f')$ and the twisted groupoid $(\widehat{G}\rtimes^f\Gc,\sigma^{\rho\nu})$ is Morita equivalent to $(\widehat{G}\rtimes^{f'}\Gc',\sigma^{\rho'\nu'})$.
\end{enumerate}.
\end{prop}
Note that we have not claimed that a Morita equivalence at groupoid level produces a Morita equivalence of Pontryagin dual groupoids.  Though there is often such a correspondence, it is not clear that there is always one.

Here are a couple of important examples of Pontryagin duality:
\begin{example}\label{ex:zuoneduality} Pontryagin duality actually shows that any twisted groupoid algebra is a $C^*$-subalgebra of the (untwisted) groupoid algebra of a $\Uone$-gerbe.  Using the notation of Theorem \eqref{T:pontryagin}, set $\rho=\nu=1$,\ $G=\Zb$, and $\widehat{G}=\Uone$.  Then the twisted groupoid $((\Zb\rtimes_{1}\Gc),\sigma^f)$ (this is a trivial $\Zb$-bundle with twisting $\sigma^{f}$) is Pontryagin dual to the gerbe $(\Uone\rtimes^f\Gc)$.  Explicitly,
\[\sigma^f((n,\gamma_1),(n,\gamma_2))=f(\gamma_1,\gamma_2)^n,\ \text{ for }((n,\gamma_1),(n,\gamma_2))\in(\Zb\rtimes_{1}\Gc)_2,\ f\in Z^2(
\Gc;\Uone) \]
But the functions in $C^*(\Uone\rtimes^f\Gc)\simeq C^*(\Zb\rtimes_1\Gc,\sigma^f)$ with support in $\{\ 1\}\rtimes_1\Gc$ clearly form a $C^*$-subalgebra identical to $C^*(\Gc;f)$.
\end{example}
\begin{example} Again in the notation of Theorem \eqref{T:pontryagin}, the case $\nu=f=1$ shows that a generalized principal bundle $G\rtimes_\rho\Gc$ is Pontryagin dual to the twisted groupoid $(\widehat{G}\rtimes^1\Gc,\tau^\rho)$.  (The latter object is the trivial $\widehat{G}$-gerbe on $\Gc$, with twisting $\tau^\rho$.)  One might wonder if this duality can be expressed purely in terms of gerbes.  The answer is that it cannot.  More precisely, this duality does not extend via the association
\[ (\text{twisted groupoids})\ \longleftrightarrow\ (\Uone\text{-gerbes}) \]
to a duality between $\Uone$-gerbes that is implemented by Fourier isomorphism of groupoid algebras.  Indeed, the gerbe associated to right side is $\Uone\rtimes^{\tau}(\widehat{G}\rtimes^1\Gc)$\ (here $\tau=\tau^\rho$), but the extension on the left side corresponding (in the sense of having a Fourier isomorphic $C^*$-algebra) to that gerbe is $G\rtimes_\chi(\Zb\rtimes_1\Gc)$, where $\chi(n,\gamma):=\rho(\gamma)^n$, which cannot be written as a $\Uone$-gerbe.  The groupoid $G\rtimes_\chi(\Zb\rtimes_1\Gc)$ actually corresponds to the disjoint union of the tensor powers of the generalized principal bundle.
\end{example}
Pontryagin duality will be used to explain ``classical'' T-duality, in Section \eqref{S:classicalT-duality}.  Before we get to T-duality, however, we need the tools to compare Morita equivalent twisted groupoids, and we need to know some specific Morita equivalences.  The development of these tools is the subject of the next two sections.

\section{Twisted Morita equivalence}\label{S:twistedMoritaEquivalence}

Let $\Hc$ and $\Kc$ be groupoids.  A Morita $(\Hc$-$\Kc)$-bimodule $P$ determines a way to compare the groupoid cohomology of $\Hc$ with that of $\Kc$.  More specifically, there is a double complex $C^{ij}(P)$ associated to the bimodule such that the moment maps
\[ \Hc_0\leftarrow P\rightarrow \Kc_0 \]
induce augmentations by the groupoid cohomology complexes
\begin{equation}\label{eq:complexinclusions}
C^i(\Hc;M)\hookrightarrow C^{i\bullet}(P;M)\ \text{ and }\ C^{\bullet j}(P;M)\hookleftarrow C^j(\Kc;M).
\end{equation}
We say groupoid cocycles $c\in C^n(\Hc;M)$ and $c'\in C^n(\Kc;M)$ are \textbf{cohomologous} when their images in $C(P;M)$ are cohomologous.  We assume for simplicity that the coefficients $M$ are constant and have the trivial actions of $\Hc$ and $\Kc$.

The double complex is defined as follows.
\[ C^{ij}(P):=(C(\Hc_j\times_{\Hc_0} P\times_{\Kc_0}\Kc_i;M);\ \delta^\Hc,\ \delta^\Kc) \]
The differentials ($\delta^\Hc:C^{ij}\to C^{ij+1}$) and ($\delta^\Kc:C^{ij}\to C^{i+1\ j}$) are given, for $f\in C^{ij}$, by the formulas

$\delta^\Hc f(h_1,\dots,h_{j+1},p,k's):=$
\begin{align*}
&f(h_2,\dots,h_{j+1},p,k's)+\sum_{n=1\dots j}(-1)^{n}f(\dots,h_nh_{n+1},\dots,p,k's)\\
       &\hspace{1.3in}+(-1)^jf(h_1,\dots,h_j,h_{j+1}p,k's)
\end{align*}

$\delta^\Kc f(h's,p,h_1\dots,h_{i+1}):=$
\begin{align*}
&f(h's,pk_1,k_2,\dots,k_i)+\sum_{n=1\dots i}(-1)^{n}f(h's,p,\dots,k_nk_{n+1},\dots)\\
       &\hspace{1.3in}+(-1)^kf(h's,p,k_1,\dots,k_i).
\end{align*}

Our main reason for comparing cohomology of Morita equivalent groupoids is the following theorem.  The first statement of the theorem is a classic statement about (stack theoretic) gerbes, but we will reproduce it for convenience.

\begin{thm}\label{p:twistedGroupoidMoritaEq} Let $P$ be a Morita equivalence $(\Hc$-$\Kc)$-bimodule, let $M$ be an Abelian group acted upon trivially by the two groupoids, and suppose we are given 2-cocycles
\[ \psi\in Z^2(\Hc;M)\text{ and }\chi\in Z^2(\Kc;M) \]
whose images $\tilde{\psi}$ and $\tilde{\chi}$ in the double complex of the Morita equivalence are cohomologous.  Then
\begin{enumerate}
\item The $M$-gerbes $M\rtimes^\psi\Hc$ and $M\rtimes^\chi\Kc$ are Morita equivalent.
\item For any group homomorphism $\phi:M\to N$, the $N$-gerbes $N\rtimes^{\phi\circ\psi}\Hc$ and $N\rtimes^{\phi\circ\chi}\Kc$ are Morita equivalent.
\item For any group homomorphism $\phi:M\to\Uone$, the twisted $C^*$-algebras $C^*(\Hc;\phi\circ\psi)$ and $C^*(\Kc;\phi\circ\chi)$ are Morita equivalent.
\end{enumerate}
\end{thm}
For example when $M=\Uone$ with $\Hc$ and $\Kc$ acting trivially, the representations of $M$ are identified with the integers, $(u\mapsto u^n)$ and the proposition implies that the ``gerbes of \emph{weight n}'', $C^*(\Hc;\psi^n)$ and $C^*(\Kc;\chi^n)$, are Morita equivalent.
\begin{proof}  We begin by proving the statement about $M$-gerbes.
The idea is that a cocycle $(\mu,\nu^{-1})\in C^{1,0}\times C^{0,1}$ satisfying
\begin{equation}\label{eq:munuboundary}
D(\mu,\nu^{-1}):=(\delta^\Hc\mu,\delta^\Kc\mu\delta^\Hc\nu,
\delta^\Kc\nu^{-1})=(\tilde{\psi},0,\tilde{\chi^{-1}})\in C^{2,0}\times C^{1,1}\times C^{0,2}\end{equation}
provides exactly the data needed to form a Morita $(M\rtimes_\psi\Hc)-(M\rtimes_\chi\Kc)$-bimodule structure on $M\times P$.  Indeed, define for $m's\in M$, $h's\in\Hc$, $k's\in\Kc$, and $p's\in P$,
\begin{enumerate}
\item Left multiplication of $M\rtimes_\psi \Hc$: $(m_1,h)*(m_2,p):=(m_1m_2\mu(h,p),hp)$
\item Right multiplication of $M\rtimes_\chi \Kc\arrows\Kc^0$: $(m_1,p)*(m_2,k):=(m_1m_2\nu(p,k),pk)$.
\end{enumerate}
Then
\begin{align*}
&\delta^\Hc\mu=\tilde{\psi}&\Leftrightarrow&\text{ the left multiplication is homomorphic }\\
&\delta^\Kc\nu=\tilde{\chi}&\Leftrightarrow&\text{ the right multiplication is homomorphic }\\
&\delta^\Kc\mu=\delta^\Kc\nu^{-1}&\Leftrightarrow&\text{ the left and right multiplications commute. }
\end{align*}
For example the equality
\[
\delta^\Kc\nu(p,k_1,k_2):=\nu(pk_1,k_2)\nu(p,k_1k_2)\nu(p,k_1)^{-1}=\tilde{\chi}(p,k_1,k_2)=:\chi(k_1,k_2)
\]
holds if and only if
\begin{align*}
((m,p)*(m_1,k_1))*(m_2,k_2)&=(mm_1m_2\nu(p,k_1)\nu(pk_1,k_2),pk_1k_2)\\
   &=(m(m_1m_2\chi(k_1,k_2))\nu(p,k_1k_2),pk_1k_2)\\
   &=(m,p)*((m_1,k_1)*(m_2,k_2)).
\end{align*}
Now we will check that the right action is principal and that the orbit space $(M\times P)\slash(M\rtimes_\chi\Kc)$
is isomorphic to $\Hc^0$.

Suppose $(m,p)*(m_1,k_1)=(m,p)*(m_2,k_2).$  Then $k_1=k_2$ since the action of $\Kc$ is principal, and then $m_1=m_2$ is forced, so the action is principal.
Next, the equation $(m,p)*(m_1\nu(p,k)^{-1},k)=(mm_1,pk)$ makes it clear that the orbit space $(M\times P)\slash(M\rtimes_\chi\Kc)$ is the same as the orbit space $P\slash\Kc$, which is $\Hc^0$.

The situation is obviously symmetric, so the left action satisfies the analogous properties, and thus the first statement of the proposition is proved.

The second statement now follows immediately because whenever Equation \eqref{eq:munuboundary} is satisfied for the quadruple $(\mu,\nu,\psi,\chi)$ in $C(P;M)$, it is also satisfied for $(\phi\circ\mu,\phi\circ\nu,\phi\circ\psi,\phi\circ\chi)$ in $C(P;N)$.  In other words $\phi\circ\psi$ and $\phi\circ\chi$ are cohomologous.

The third statement of the theorem can be proved directly by exhibiting a $C^*$-algebra bimodule, but this will not be necessary because we already have a Morita $C^*(M\rtimes^\psi\Hc)$-$C^*(M\rtimes^\chi\Kc)$-bimodule, coming from Proposition \eqref{P:moritaFacts} and the fact that $M\rtimes^\psi\Hc$ and $M\rtimes^\chi\Kc$ are Morita equivalent.  We will manipulate this bimodule into a $C^*(\Hc,\phi\circ\psi)$-$C^*(\Kc,\phi\circ\chi)$-bimodule.

Note that $(\widehat{M}\rtimes_1\Hc,\sigma^\psi)$ is Pontryagin dual to $(M\rtimes^\psi\Hc)$ and $(\widehat{M}\rtimes_1\Kc,\sigma^\chi)$ is Pontryagin dual to $(M\rtimes^\chi\Kc)$, thus the associated $C^*$-algebras are pairwise Fourier isomorphic.  Then the Morita equivalence bimodule (which is a completion of $C_c(M\times P)$, where  $P$ is the $\Hc$-$\Kc$-bimodule) for $C^*(M\rtimes^\psi\Hc)$ and $C^*(M\rtimes^\chi\Kc)$ is taken via the Fourier isomorphism to a Morita equivalence bimodule between $C^*(\widehat{M}\rtimes_1\Hc,\sigma^\psi)$ and $C^*(\widehat{M}\rtimes_1\Kc,\sigma^\chi)$.  This Fourier transformed bimodule will be a completion $X$ of $C_c(\widehat{M}\times P)$.  Then for $\phi\in\widehat{M}$, evaluation at $\phi$ determines projections
\[ \ev_\phi:C^*(\widehat{M}\rtimes_1\Hc,\sigma^\psi)\to C^*(\Hc,\phi\circ\psi), \]
\[ \ev_\phi:C^*(\widehat{M}\rtimes_1\Kc,\sigma^\chi)\to C^*(\Kc,\phi\circ\chi), \]
\[ \text{ and }\ev_\phi:C_c(\widehat{M}\times P)\to C_c(P) \]
that are compatible with the bimodule structure of $X$, so $\ev_\phi(X)$ is automatically a Morita equivalence $C^*(\Hc,\phi\circ\psi)$-$C^*(\Kc,\phi\circ\chi)$-bimodule.  Thus $X$ is actually a family of Morita equivalences parameterized by $\phi\in\widehat{M}$, and in particular the third statement of the theorem is true.
\end{proof}
\begin{rem}Just as a Morita equivalence lets you compare cohomology, a $G$-equivalence lets you compare $G$-equivariant cohomology.  Indeed, all complexes involved will have the commuting $G$-actions, and there will be an associated tricomplex which is coaugmented by the complexes computing equivariant cohomology of the two groupoids.  Rather than using the tricomplex, however, one can map the equivariant complexes into the complexes of the crossed product groupoids (as in Equation \eqref{Eq:equicohoCrossedCoho}) and do the comparing there.  The result is the same but somewhat less tedious to compute.
\end{rem}

\section{Some facts about groupoid cohomology}\label{S:cohoFacts}

Now there is good motivation for wanting to know when groupoid cocycles are cohomologous.  In this section we will collect some facts that help in that pursuit.

The Morita equivalences that come from actual groupoid homomorphisms have nice properties with respect to cohomology. They are called essential equivalences.
\begin{defn}\cite{C} Let $\phi:\Gc\to\Hc$ be a morphism of topological groupoids (i.e. a continuous functor).  This morphism determines a right principal $\Gc$-$\Hc$-bimodule $P_\phi:=\Gc_0\times_{\Hc_0}\Hc_1$.  If $P_\phi$ is a Morita equivalence bimodule (which follows if it is left principal) then $\phi$ is called an {\em essential equivalence.}
\end{defn}
Note that for any morphism $\phi$ the left moment map $P_\phi\to\Gc_0$ has a canonical section.
\begin{prop}\label{P:nicegroupoidFacts} Let $\Gc$ and $\Hc$ be topological groupoids, and suppose $P$ is a Morita equivalence $\Gc$-$\Hc$-bimodule with left and right moment maps $\Gc_0\stackrel{\ell}{\leftarrow}P\stackrel{\rho}{\rightarrow}\Hc_0$.  Then:
\begin{enumerate}
\item $P$ is equivariantly isomorphic to $P_\phi$ for some essential equivalence $\phi:\Gc\to\Hc$ if and only if the left moment map $\ell:P\to\Gc_0$ admits a continuous section.
\item Any morphism $\phi:\Gc\to\Hc$ determines, via pullback, a chain morphism $\phi^*:C^\bullet(\Hc)\to C^\bullet(\Gc)$.
\item The moment maps $\ell$ and $\rho$ determine chain morphisms
\[ C^\bullet(\Gc)\stackrel{\ell^*}{\longrightarrow}\tot(C^{\bullet\bullet}(P))
   \stackrel{\rho^*}{\longleftarrow}C^\bullet(\Hc). \]
\item Any continuous section of $\ell:P\to\Gc_0$ determines a contraction of the coaugmented complex $C^k(\Gc)\to C^{\bullet k}(P)$ for each $k$, and thus a quasi-inverse $[\ \ell^*]^{-1}:H^*(\tot(C(P))\to H^*(\Gc)$, and thus a homomorphism
\[ [\ \ell^*]^{-1}\circ[\ \rho^*]:H^*(\Hc)\to H^*(\Gc). \]
\item The two chain morphisms $\ell^*\circ\phi^*$ and $\rho^*$ are homotopic; in particular
\[ [\ \phi^*]=[\ \ell^*]^{-1}\circ[\ \rho^*].   \]
\end{enumerate}
\end{prop}
\begin{proof}
The first statement is an easy exercise and the next two statements do not require proof.  The homotopy for the fourth statement is in the proof of Lemma 1 \cite{C} and the homotopy for the fifth statement is written on page (8) of \cite{C}.
\end{proof}

\begin{cor}\label{C:fundamental} Suppose that $\phi:\Gc\to\Hc$ is an essential equivalence, that $M$ is a locally compact abelian group viewed as a trivial module over both groupoids, and that $\chi\in Z^2(\Hc;M)$ is a 2-cocycle.  Then $\chi$ and $\phi^*\chi$ satisfy the conditions of Theorem \eqref{p:twistedGroupoidMoritaEq}; in particular $M\rtimes^\chi\Hc$ is Morita equivalent to $M\rtimes^{\phi^*\chi}\Gc$.
\end{cor}
\begin{proof}  Use the notation of Proposition \eqref{P:nicegroupoidFacts}.  The images of $\chi$ and $\phi^*\chi$ in $C^{\bullet\bullet}(P)$ are $\ell^*\circ\phi^*\chi$ and $\rho^*\chi$ respectively, which are cohomologous by statement $(5)$, and these are precisely the conditions of Theorem \eqref{p:twistedGroupoidMoritaEq}.
\end{proof}
\begin{cor}\label{C:fundamental2} If $\phi:\Gc\to\Hc$ is an essential equivalence such that the right moment map $P_\phi\stackrel{\rho}{\to}\Hc_0$ admits a section then $[\phi^*]:H^*(\Hc)\to H^*(\Gc)$ is an isomorphism.
\end{cor}
\begin{proof}  This follows from Corollary \eqref{C:fundamental} and Part (4) of Proposition \eqref{P:nicegroupoidFacts}.
\end{proof}
Corollaries \eqref{C:fundamental} and \eqref{C:fundamental2} will be very useful because all of the Morita equivalences that have been introduced so far are essential equivalences, as the next proposition shows.  Before the proposition let us describe two more groupoids.
\begin{example}\label{Ex:bestgroupoid1} Let $\rho:\Gc\to G/N$ be a homomorphism and let $G\times_{G/N,\rho}\Gc_1$ denote the fibred product (i.e. the space $\{(g,\gamma)\ |\ gN = \rho(\gamma)\in G/N\}$).  Define
\[ G\times_{G/N,\rho}\Gc:= (G\times_{G/N,\rho}\Gc_1\arrows\Gc_0) \]
with structure maps $s(g,\gamma):=s\gamma$, $r(g,\gamma):=r\gamma$, and $(g_1,\gamma_1)\circ(g_2,\gamma_2):=(g_1g_2,\gamma_1\gamma_2)$.  Note that any lift of $\rho$ to a continuous map $\tilde\rho:\Gc_1\to G$ determines an isomorphism of groupoids
\[ N\rtimes^{\delta\rho}\Gc\longrightarrow G\times_{G/N,\rho}\Gc\qquad (n,\gamma)\mapsto(n\tilde\rho(\gamma),\gamma) \]
When no such lift exists this fibred product groupoid is an example of an $N$-gerbe without section.
\end{example}
\begin{example}\label{Ex:bestgroupoid2}
The fibred product groupoid of Example \eqref{Ex:bestgroupoid1} is equipped with a canonical right module $P:=G\times\Gc_0$, for which the induced groupoid is
\[ G\rtimes(G\times_{G/N,\rho}\Gc):=((G\times G\times_{G/N,\rho}\Gc_1)\arrows G\times\Gc_0) \]
with structure maps $s(g,g',\gamma):=(gg',s\gamma)$, $r(g,g',\gamma):=(g,r\gamma)$, and $(g,g_1,\gamma_1)\circ(gg_1,g_2,\gamma_2):=(g,g_1g_2,\gamma_1\gamma_2)$.  A lift $\tilde\rho$
determines a homeomorphism
\[ G\rtimes (N\rtimes^{\delta\rho}\Gc)\longrightarrow G\rtimes(G\times_{G/N,\rho}\Gc)\qquad (g,n,\gamma)\mapsto(g,n\tilde\rho(\gamma),\gamma) \]
which implicitly defines the groupoid structure on $G\rtimes N\rtimes^{\delta\rho}\Gc$ as the one making this a groupoid isomorphism.
\end{example}
\begin{prop}[\textbf{ Essential equivalences }]\label{P:essentialEquivs}Let $\Gc$ be a groupoid, $N$ a closed subgroup of a locally compact group $G$, $\Nc_G(N)$ its normalizer in $G$, and $\rho:\Gc\to\Nc_G(N)/N\subset G/N$ a homomorphism.  Then
\begin{enumerate}
\item The gluing morphism $\Gc_\Uf\to\Gc$ from a refinement $\Gc_\Uf$ corresponding to a locally finite cover $\Uf$ of $\Gc_0$ (see Example \eqref{ex:cechgroupoid}) is an essential equivalence.
\item The quotient morphism
\[ X\rtimes N \to (X/N\arrows X/N)\qquad (x,n)\mapsto xN\in X/N, \]
for $X$ a free and proper right $N$-space, is an essential equivalence.
\item The inclusion
\[ \iota:(G\times_{G/N,\rho}\Gc)\hookrightarrow G\ltimes G/N\rtimes_\rho\Gc\qquad (g,\gamma)\mapsto (g,eN,\gamma)\]
is an essential equivalence and in particular if $\rho$ lifts to a continuous map $\tilde\rho:\Gc\to G$ then
\[ \iota:(N\rtimes^{\delta\rho}\Gc)\hookrightarrow G\ltimes(G/N\rtimes_\rho\Gc)
     \qquad (n,\gamma)\mapsto (n\tilde\rho(\gamma),eN,\gamma)  \]
is an essential equivalence.
\item The quotient map
\[ \kappa:G\rtimes(G\times_{G/N,\rho}\Gc)\to G/N\rtimes_\rho\Gc\qquad (g,g',\gamma)\mapsto (gN,\gamma)\]
is an essential equivalence and in particular if $\rho$ lifts to a continuous map $\tilde\rho:\Gc\to G$ then
\[ \kappa:G\rtimes(N\rtimes^{\delta\rho}\Gc)\to G/N\rtimes_\rho\Gc\qquad (g,n,\gamma)\mapsto (gN,\gamma)\]
is an essential equivalence.
\item If $\Gc$ is a \v{C}ech groupoid, $N$ is normal in $G$, and
\[ Q:= (G/N\times\Gc_0)/(t,r\gamma)\sim(t\rho(\gamma),s\gamma)\]
is the principle $G/N$-bundle on $X=\Gc_0/\Gc_1$ with transition functions given by $\rho$, then the quotient
\[ q:(G/N\rtimes_\rho\Gc)\longrightarrow (Q\arrows Q) \]
is an essential equivalence.
\end{enumerate}
Furthermore, in all cases where there are obvious left $G$ actions, these are essential $G$-equivalences.
\end{prop}
\begin{proof}
For the groupoids that were introduced in Section \eqref{S:Examples}, simply check that the bimodules determined by these morphisms are the same as the Morita equivalence bimodules that were described there.  The two new cases involving Examples \eqref{Ex:bestgroupoid1} and \eqref{Ex:bestgroupoid2} are very similar and are left for the reader.  The statement about $G$-equivariance is clear.
\end{proof}
\begin{prop}\label{P:cohoIsoms} In the notation of Proposition \eqref{P:essentialEquivs}, suppose that $G\to G/N$ admits a continuous section (for example if $N=\{e\}$, or if $N$ is a component of $G$ or if $G$ is discrete), then the essential equivalences $\iota$ and $\kappa$ induce isomorphisms of groupoid cohomology.
\end{prop}
\begin{proof} By Corollary \eqref{C:fundamental2}, we only need to produce sections of the right moment maps, and the section of $\sigma:G/N\to G$ provides these.  The case $\Gc=*$ illustrates the answer: for $\iota$,  $G/N\ni gN\mapsto(*,\sigma(gN),eN)\in \{*\}\times G\times{eN}\simeq P_\iota$ does the job, while for $\kappa$ it is $G/N\ni gN\mapsto (\sigma(gN),gN)\in G\times_{G/N}G/N\simeq P_\kappa$.
\end{proof}
The following proposition is important because it implies that cocycles on crossed product groupoids can be assumed to be in a special form.
\begin{prop} The groupoid cohomology $H^*(G\ltimes(G\rtimes_\rho\Gc);B))$ is a direct summand of the equivariant cohomology $H^*_G(G\rtimes_\rho\Gc;B)$.
\end{prop}
\begin{proof} We will show that the identity map on $H^*(G\ltimes(G\rtimes_\rho\Gc);B))$ factors through $H^*_G(G\rtimes_\rho\Gc;B)$.

The inclusion $\iota:\Gc\hookrightarrow G\ltimes(G\rtimes_\rho\Gc)$ sending $\gamma\mapsto(\rho(\gamma),e,\gamma)$ induces a chain map
\[ \iota^*:C^\bullet(G\ltimes(G\rtimes_\rho\Gc);B)\To C^\bullet(\Gc;B) \]
which is a quasi-isomorphism by Proposition \eqref{P:cohoIsoms}.  The quotient $G\ltimes(G\rtimes_\rho\Gc)\To\Gc$ induces the chain map
\[ q^*:C^\bullet(\Gc;B)\To C^\bullet(G\ltimes(G\rtimes_\rho\Gc);B) \]
which is also a quasi-isomorphism, and in fact induces the quasi-inverse to $\iota^*$ since $\iota^*\circ q^*=(q\circ\iota)^*=Id^*$.

Now the quotient $G\rtimes_\rho\Gc\to\Gc$ induces a chain map $C^\bullet(\Gc)\to C^\bullet(G\rtimes_\rho\Gc)$ and thus a chain map
\[  C^\bullet(\Gc)\to C^\bullet(G\rtimes_\rho\Gc)\to\tot C^\bullet(G,C^\bullet(G\rtimes_\rho\Gc)). \]
Note that the first map and the composition of both maps are chain morphisms, but the second is not.
Following this by the morphism $\tot C^\bullet(G,C^\bullet(G\rtimes_\rho\Gc))\to C^\bullet(G\ltimes(G\rtimes_\rho\Gc))$ of Equation \eqref{Eq:equicohoCrossedCoho} induces a sequence
\[ C^\bullet(\Gc)\To \tot C^\bullet(G;C^\bullet(G\rtimes_\rho\Gc))\To C^\bullet(G\ltimes(G\rtimes_\rho\Gc)) \]
whose composition is easily seen to equal $q^*$.  Finally, precomposing with $\iota^*$ provides the promised factorization.
\end{proof}
\begin{rem} Whenever a group $G$ acts freely and properly on a groupoid $\Hc$ and $\Hc\to G\backslash\Hc$ has local sections, then $\Hc$, being a principal $G$-bundle, is equivariantly Morita equivalent to $G\rtimes_\rho\Gc$ for some $\rho$ and $\Gc$, where $\Gc$ is a refinement of the quotient groupoid $\Hc/G:=(\Hc_1/G\arrows\Hc_0/G)$.  In particular there is a refinement of $\Hc$ that is equivariantly isomorphic to $G\rtimes_\rho\Gc$.  Thus after a possible refinement the above Proposition is a statement about the cohomology of all free and proper $G$-groupoids with local sections.
\end{rem}

For completeness we will include the following proposition, which might be attributable to Haefliger.  It implies that one can work exclusively with essential equivalences if desired.
\begin{prop} Let $P$ be a $\Gc$-$\Hc$-Morita equivalence which is locally trivial as an $\Hc$-module, then the Morita equivalence can be factored in the form:
\[ \Gc\longleftarrow \Gc_\Uf\stackrel{\phi}{\longrightarrow}\Hc \]
where $\Gc_\Uf$ is a refinement of $\Gc$ and $\phi$ is an essential equivalence.
\end{prop}
\begin{proof} The left moment map for $P$ admits local sections so there is a cover $\Uf$ of $\Gc_0$ such that the refined Morita $\Gc_\Uf$-$\Hc$-bimodule $P_\Uf$ has a section for its left moment map.  By Proposition \eqref{P:nicegroupoidFacts} $P_\Uf\simeq P_\phi$ for some essential equivalence $\phi$.
\end{proof}

\section{Generalized Mackey-Rieffel imprimitivity}
In this section we show that a specific pair of twisted groupoids is Morita equivalent.  The two Morita equivalent twisted groupoids correspond, respectively, to a $\Uone$-gerbe on a crossed product groupoid for a $G$-action on a generalized principal $G/N$-bundle, and to a $\Uone$-gerbe over an $N$-gerbe.  We also describe group actions on the groupoids that make the Morita equivalence equivariant.  The equivalence is a simple consequence of the methods of Section \eqref{S:classicalT-duality}), but it deserves to be singled out because  both twisted groupoids appear in the statement of classical T-duality.

\begin{thm}[\textbf{Generalized Mackey-Rieffel imprimitivity}]\label{p:GenRieffelImprim} Let $\Gc$ be a groupoid, $G$ a locally compact group, $N<G$ a closed normal subgroup, and $\bar{\rho}:\Gc\to G/N$ a homomorphism that admits a continuous lift $\rho:\Gc\to G$.  Form the two groupoids of Example \eqref{ex:bundleimprimitivity}:
\begin{enumerate}
\item $\Hc:=G\ltimes(G/N\rtimes_{\bar{\rho}}\Gc)$
\item $\Kc:=N\rtimes^{\delta\rho}\Gc$
\end{enumerate}
Then for any $G$-equivariant gerbe presented by a 2-cocycle $(\sigma,\lambda,\beta)$ as in Equation \eqref{eq:twococycle}, there is a Morita equivalence of twisted groupoids
\[ (\Hc,\psi)\sim(\Kc,\chi) \]
where $\psi\in Z^2(\Hc;\Uone)$ and $\chi\in Z^2(\Kc;\Uone)$ are the given by
\begin{align}\label{Eq:psichi}
\psi((g_1,t,\gamma_1),(g_2,g_1^{-1}t\rho(\gamma_1),\gamma_2))
&:=\sigma(t,\gamma_1,\gamma_2)\lambda(g_1,t\rho(\gamma_1),\gamma_2)\notag \\
&\qquad\times\beta(g_1,g_2,t\rho(\gamma_1)\rho(\gamma_2),s\gamma_2)\\
\chi((n_1,\gamma_1),(n_2,\gamma_2))
&:=\sigma(e_{G/N},\gamma_1,\gamma_2)\lambda(n_1\rho(\gamma_1),\rho(\gamma_1),\gamma_2)\notag\\
&\qquad\times\beta(n_1\rho(\gamma_1),n_2\rho(\gamma_2),\rho(\gamma_1)\rho(\gamma_2),s\gamma_2),
\end{align}
for $g's\in G$, $t\in G/N$, $n's\in N$, and $\gamma's\in\Gc$.
\end{thm}

\begin{proof}
For the sake of understanding, let us first see where $\psi$ comes from.  Set $(g_1,h_1)=(g_1,t,\gamma_1),(g_2,h_2)=(g_2,g_1^{-1}t\rho(\gamma_1),\gamma_2)$.  Then a composed pair looks like $(g_1,h_1)(g_2,h_2)=(g_1g_2,h_1g_1h_2)$, and $\psi=\sigma(h_1,g_1h_2)\lambda(g_1,g_1h_2)\beta(g_1,g_2,g_1g_2(sh_2))$.  In other words, $\psi=F(\sigma,\lambda,\beta)$, the image of $(\sigma,\lambda,\beta)$ under the chain map
\[F:\tot C^\bullet(G,C^\bullet((G/N\rtimes_{\bar{\rho}}\Gc),\Uone))\to C^\bullet(\Hc;\Uone)\]
of Equation \eqref{Eq:equicohoCrossedCoho}.  So $\psi$ comes from extending a 2-cocycle from $(G/N\rtimes_{\bar{\rho}}\Gc)$ to an equivariant 2-cocycle. Now $\chi=\iota^*\circ\psi$, where $\iota:N\rtimes^{\delta\rho}\Gc\to G\ltimes G/N\rtimes_{\bar{\rho}}\Gc$ is as in Proposition \eqref{P:essentialEquivs} and the theorem follows immediately from the results in Section \eqref{S:cohoFacts}.
\end{proof}
\begin{rem} In the hypotheses of the above theorem it is not necessary to have the lift $\tilde{\rho}$.  In the absence of such a lift, one simply replaces $N\rtimes^{\delta\rho}\Gc$ by $G\times_{G/N,\rho}\Gc$ as in Example \eqref{Ex:bestgroupoid1}.
\end{rem}
When $G$ is abelian there is a canonical $\widehat{G}$ action, denoted $\hat{\alpha}$, on $\Uone\rtimes^\psi\Hc$, given by
\begin{equation}\label{E:GhatAction}
\phi\cdot(\theta,g,t,\gamma)=(\theta\langle\phi,g\rangle,g,t,\gamma),\quad \phi\in\widehat{G},\ (\theta,g,t,\gamma)\in\Uone\rtimes^\psi\Hc.
\end{equation}
This action corresponds to the natural $\widehat{G}$-action on a crossed product algebra $G\ltimes A$.
As one expects, the above Morita equivalence takes this action to the same action, pulled back via $\iota$.
\begin{prop}\label{P:actionsforGenRiefImp} In the notation of Theorem \eqref{p:GenRieffelImprim} let $G$ be an abelian group.  Then the canonical $\widehat{G}$-action, $\hat{\alpha}$, on $\Uone\rtimes^\psi\Hc$ is transported under the Morita equivalence $(\Uone\rtimes^\psi\Hc)\sim(\Uone\rtimes^\chi\Kc)$ to
\[ \iota^*(\hat{\alpha})(\phi)(\theta,n,\gamma):=(\theta\langle\phi,n\rho(\gamma)\rangle,n,\gamma),\quad (\theta,n,\gamma)\in\Uone\rtimes^\chi\Kc. \]
Thus with these actions the Morita equivalence of Theorem \eqref{p:GenRieffelImprim} becomes $\widehat{G}$-equivariant.
\end{prop}
\begin{proof} $\Uone\times P_\iota$ is the Morita $(\Uone\rtimes^\psi\Hc)$-$(\Uone\rtimes^{\chi}\Kc)$-bimodule here.  For $\phi\in\widehat{G}$ and $(\theta,g,\gamma)\in\Uone\times P_\iota$, define an action by $\phi\cdot(\theta,g,\gamma)=(\theta\langle\phi,g\rangle,g,\gamma)$.  Then this action, along with the actions $\hat{\alpha}$ and $\iota^*\hat{\alpha}$ determines an equivariant Morita equivalence (that is, Equation \eqref{eq:GmoritaEquiv} is satisfied for these actions).
\end{proof}

\section{Classical T-duality}\label{S:classicalT-duality}
Let us start with the definition.  For us \textbf{classical T-duality} will refer to a T-duality between a $\Uone$-gerbe on a generalized torus bundle and a $\Uone$-gerbe on a generalized principal dual-torus bundle.  If instead of a torus bundle we start with a $G/N$-bundle, for $N$ a closed subgroup of an abelian locally compact group $G$, then we still call this classical T-duality since it is the same phenomenon.  Thus ``classical'' means for us that the groups involved are abelian and furthermore that the dual object does not involve noncommutative geometry.

We are now ready to fully describe the T-dualization procedure.  There will be several remarks afterwards.

\noindent\textbf{Classical T-duality.}
Suppose we are given the data of a generalized principal $G/N$-bundle and $G$-equivariant twisting 2-cocycle whose $C^{2,0}$-component $\beta\in C^2(G;C^0(G/N\rtimes_{\bar{\rho}}\Gc;\Uone))$ is trivial:
\begin{equation}\label{E:classicalTdata}
(\ G/N\rtimes_{\bar{\rho}}\Gc,\ (\sigma,\lambda,1)\in Z^2_G(G/N\rtimes_{\bar{\rho}}\Gc;\Uone)\ ).
\end{equation}
Then this is the initial data for a classically T-dualizable bundle, and the following steps produce its T-dual.
\begin{rem} Because $G$ is abelian the restriction of $\lambda:G\times G/N\times\Gc_1\to\Uone$ to $N\times G/N\times\Gc_1$ does not depend on $G/N$.  Indeed, for $n\in N$, $g\in G$, and $t\in G/N$
\[ \lambda(n,t,\gamma)=\lambda(ng,t,\gamma)\lambda(g,t,\gamma)^{-1}=
                  \lambda(gn,t,\gamma)\lambda(g,t,\gamma)^{-1}=\lambda(n,g^{-1}t,\gamma).\]
Furthermore, the cocycle conditions ensure that $\lambda|_{N\times G/N\times\Gc_1}$ is homomorphic in both $N$ and $\Gc$.  We write $\bar{\lambda}$ for the induced homomorphism $\bar{\lambda}:\Gc\to\widehat{N}$.\end{rem}
\ \newline
\noindent\underline{Step 1}\ \ Pass from $(G/N\rtimes_{\bar{\rho}}\Gc,\ (\sigma,\lambda,1))$ to the crossed product groupoid $G\ltimes(G/N\rtimes_{\bar{\rho}}\Gc)$, with twisting 2-cocycle $\psi:=F(\sigma,\lambda,1)$ of Equation \eqref{Eq:psichi}, and with canonical $\widehat{G}$-action $\hat{\alpha}$ of Equation \eqref{E:GhatAction}:
\[(\ G\ltimes(G/N\rtimes_{\bar{\rho}}\Gc),\ \psi,\ \hat{\alpha}\ ).\]
\ \newline
\noindent\underline{Step 2}\ \ Choose a lift $\rho:\Gc\to G$ of $\bar{\rho}$ and pass from $G\ltimes(G/N\rtimes_{\bar{\rho}}\Gc)$ to the Morita equivalent $N$-gerbe $N\rtimes^{\delta\rho}\Gc$ with twisting 2-cocycle $\iota^*\psi$ and $\widehat{G}$-action $\iota^*\hat{\alpha}$ of Proposition \eqref{P:actionsforGenRiefImp}:
\[ (\ N\rtimes^{\delta\rho}\Gc,\ \iota^*\psi,\ \iota^*\hat{\alpha}\ ). \]
\ \newline
\noindent\underline{Step 3}\ \ Pass to the Pontryagin dual system.  More precisely, pass to the twisted groupoid with $\widehat{G}$-action whose twisted groupoid algebra is $\widehat{G}$-equivariantly isomorphic to the algebra of Step 2, when the chosen isomorphism is Fourier transform in the $N$-direction.  The result is an $\widehat{N}$-bundle, but not exactly of the form we have been considering.  Here is the Pontryagin dual system:
\begin{itemize}
\item Groupoid: $\Kc:=(\widehat{N}\times\Gc_1\arrows\widehat{N}\times\Gc_0)$
\item Source and range: $s(\phi,\gamma):=(\phi,s\gamma),\
      r(\phi,\gamma):=\phi(\bar{\lambda}(\gamma)^{-1},r\gamma)$ for $(\phi,\gamma)\in\Kc_1$
\item Multiplication:\ $(\phi\bar{\lambda}(\gamma_2)^{-1},\gamma_1)(\phi,\gamma_2):=(\phi,\gamma_1\gamma_2)$, where $\bar{\lambda}:=\lambda|_{N\times\{1\}\times\Gc_1}$ (which is homomorphic in $N$), viewed as a map $\bar{\lambda}:\Gc_1\to\widehat{N}$.
\item Twisting: $\tau(\phi(\bar{\lambda}(\gamma_2)^{-1},\gamma_1),(\phi,\gamma_2))
      :=\sigma(e,\gamma_1,\gamma_2)\lambda(\rho(\gamma_1),\gamma_2)
      \langle\phi,\delta\rho(\gamma_1,\gamma_2)\rangle.$
\item $\widehat{G}$-action:\ $\phi\cdot a(\phi',\gamma):= \langle\phi',\rho(\gamma)\rangle a(\phi^{-1}\phi',\gamma)$ for $\phi'\in\widehat{G}$ and $a\in C_c(\Kc_1)$.
\end{itemize}

For aesthetic reasons, it is preferable to put this data in the same form as Equation \eqref{E:classicalTdata}.  To do this first note that $\Kc$ is isomorphic to $\widehat{N}\rtimes_{\bar{\lambda}}\Gc$ via:
\[ \widehat{N}\rtimes_{\bar{\lambda}}\Gc\stackrel{\sim}{\longrightarrow}\Kc,\qquad (\phi,\gamma)\mapsto(\phi\bar{\lambda}(\gamma),\gamma) \]
Using this isomorphism to import the twisting and $\widehat{G}$-action on $N\rtimes_{\bar{\lambda}}\Gc$ from $\Kc$ determines that $N\rtimes_{\bar{\lambda}}\Gc$ must have:
\begin{itemize}
\item Twisting: $\sigma^\vee(\phi,\gamma_1,\gamma_2):=
       \sigma(e,\gamma_1,\gamma_2)\lambda(\rho(\gamma_1),\gamma_2)
      \langle(\phi\bar{\lambda}(\gamma_1)\bar{\lambda}(\gamma_2)),\delta\rho(\gamma_1,\gamma_2)\rangle.$
\item $\widehat{G}$-action:\ $\phi'\cdot a(\phi,\gamma):= \langle\phi,\rho(\gamma)\rangle a(\phi'^{-1}\phi,\gamma)$ for $\phi'\in\widehat{G}$ and $a\in C_c(N\rtimes_{\bar{\lambda}}\Gc)$.
\end{itemize}
But this is again classical T-duality data, indeed it is what we write as:
\[ (\ \widehat{N}\rtimes_{\bar{\lambda}}\Gc,\ (\sigma^\vee,\ \rho,\ 1)\in Z^2_{\widehat{G}}(\widehat{N}\rtimes_{\bar{\lambda}}\Gc;\Uone)\ ) \]
\ \newline
Thus the classical T-dual pair is
\begin{equation}\label{E:classicalDuality}
(\ G/N\rtimes_{\bar{\rho}}\Gc,\ (\sigma,\lambda,1)\ )\longleftrightarrow
     (\widehat{N}\rtimes_{\bar{\lambda}}\Gc,\ (\sigma^\vee,\rho,1)\ ).
\end{equation}

\noindent\underline{Some properties of the duality}
\ \newline\newline
\noindent (A)\quad  Taking $C^*$-algebras everywhere, the dualizing process becomes:
\[ C^*(G/N\rtimes_{\bar{\rho}}\Gc;\sigma)\rightsquigarrow G\ltimes_\lambda
    C^*(G/N\rtimes_{\bar{\rho}}\Gc;\sigma)\stackrel{\text{Morita}}
    {\sim}C^*(N\rtimes^{\delta\rho}\Gc;\iota^*\psi)\stackrel{iso}{\simeq} C^*(\widehat{N}\rtimes_{\bar{\lambda}}\Gc;\sigma^\vee) \]
All algebras to the right of the ``$\rightsquigarrow$'' have canonically $\widehat{G}$-equivariantly isomorphic spectra.  For the passage from Step (1) to Step (2) this follows because it is a Morita equivalence of twisted groupoids by Theorem \eqref{p:GenRieffelImprim}, and by Proposition \eqref{P:actionsforGenRiefImp} this Morita equivalence is $\widehat{G}$-equivariant.  For the passage from Step (2) to Step (3) it follows because Pontryagin dualization induces an equivariant isomorphism of twisted groupoid algebras with $\widehat{G}$-action (Theorem \eqref{T:pontryagin}).
\ \newline\newline
\noindent (B)\quad  The passages (1)$\to$(2) and (2)$\to$(3) induce isomorphisms in $K$-theory for any $G$, since $K$-theory is invariant under Morita equivalence and isomorphism of $C^*$-algebras.  Thus whenever $G$ satisfies Connes'-Thom isomorphism (i.e. $G$ satisfies $K(A)\simeq K(G\ltimes A)$ for every $G$-$C^*$-algebra $A$) the duality of \eqref{E:classicalDuality} incorporates an isomorphism of twisted K-theory:
\[ K^\bullet(G/N\rtimes_{\bar{\rho}}\Gc,\sigma)\simeq K^{\bullet+dimG}(\widehat{N}\rtimes_{\bar{\lambda}}\Gc,\sigma^\vee). \]
The class of groups satisfying Connes'-Thom isomorphism include the (finite dimensional) 1-connected solvable Lie groups.
\ \newline\newline
\noindent (C)\quad  When $\Gc$ is a \v{C}ech groupoid for a space $X$, this duality can be viewed as a duality
\[  (P\to X, [(\sigma,\lambda,1)]\in H^2_G(P;\Uone))\longleftrightarrow(P^\vee\to X,[(\sigma^\vee,\rho,1)]\in H^2_{\widehat{G}}(P^\vee;\Uone)) \]
where $P$ is a principal $G/N$-bundle and $P^\vee$ is a principal $\widehat{N}$-bundle.
Indeed, the pair $(P,[(\sigma,\lambda)])$ are the spectrum (with its $G/N$-action) and $G$-equivariant Dixmier-Douady invariant, respectively, of the $G$-algebra $C^*(G/N\rtimes_{\bar{\rho}}\Gc;\sigma)$, and it is enough to show that the spectrum and equivariant Dixmier-Douady invariant of the dual, $C^*(\widehat{N}\rtimes_{\bar{\lambda}}\Gc,\sigma^\vee)$, is independent of the groupoid presentation of $(P,[(\sigma,\lambda)])$.  But since every procedure to the right of the ''$\rightsquigarrow$'' is an equivariant Morita equivalence of $C^*$-algebras, the result will follow as long as two different presentations give rise to objects which are equivariantly Morita equivalent at Step (1).  But if $(G/N\times_{\bar{\rho}}\Gc,(\sigma,\lambda,1))$ and $(G/N\times_{\bar{\rho'}}\Gc',(\sigma',\lambda',1))$ are $G$-equivariantly Morita equivalent groupoids with cohomologous cocycles then the groupoids
$\Uone\rtimes^\psi(G\ltimes G/N\times_{\bar{\rho}}\Gc)$ and $\Uone\rtimes^{\psi'}(G\ltimes G/N\times_{\bar{\rho}'}\Gc')$ are Morita equivalent as $\widehat{G}$ groupoids, and from this the result follows immediately.

\section{Nonabelian Takai duality}\label{S:nonabTakai}
In describing classical T-duality it was crucial that the group $G$ be abelian because the dual side was viewed as a $\widehat{G}$-space, and the procedure of T-dualizing was run in reverse by taking a crossed product by the $\widehat{G}$-action.  Since our goal is to describe an analogue of T-duality that is valid for bundles of nonabelian groups, we need a method of returning from the dual side to the original side that does not involve a Pontryagin dual group.  A solution is provided by what we call nonabelian Takai duality for groupoids.  In this section we first review classical Takai duality, then we describe the nonabelian version. The nonabelian Takai duality is constructed so that when applied to abelian groups it reduces to what is essentially the Pontryagin dual of classical Takai duality.

Recall that Takai duality for abelian groups is the passage
\[ (G-C^*\text{-algebras})\rightsquigarrow(\widehat{G}-C^*\text{-algebras}) \]
\[(\alpha,A)\longmapsto(\hat{\alpha},G\ltimes_\alpha A),\]
where $\hat{\alpha}$ is the canonical $\widehat{G}$-action $\hat{\alpha}_\phi\cdot a(g):=\phi(g)a(g)$ for $\phi\in\widehat{G},\ g\in G$ and $a\in G\ltimes A$.  This is a duality in the sense that a second application  produces a $G$-algebra $(\hat{\hat{\alpha}},\widehat{G}\ltimes_{\hat{\alpha}}G\ltimes_\alpha A)$, which is $G$-equivariantly Morita equivalent to the original $(\alpha,A)$.

Now let $(\alpha,\Hc)$ be a $G$-groupoid.  Takai duality applied (twice) to the associated groupoid algebra is the passage
\begin{align*}
(\alpha,C^*(\Hc))&\longmapsto(\hat{\alpha},C^*(G\ltimes_\alpha\Hc)) \\
&\longmapsto(\hat{\hat{\alpha}},\widehat{G}\ltimes_{\hat{\alpha}}C^*(G\ltimes_\alpha\Hc)).
\end{align*}
Comparing multiplications, one sees that the last algebra is identical to the groupoid algebra of the twisted groupoid $(\widehat{G}\ltimes_{triv}(G\ltimes_\alpha\Hc),\chi)$ where
\[\chi((\phi_1,g_1,\gamma),(\phi_2,g_2,\gamma_2)):=\langle\phi_1,g_2\rangle \]
and $(triv)$ denotes the trivial action of $\widehat{G}$.
Note that this duality cannot be expressed purely in terms of groupoids since the dual group $\widehat{G}$ only acts on the groupoid algebra.  However, taking the Fourier transform in the $\widehat{G}$-direction determines a Pontryagin duality between $(\widehat{G}\ltimes_{triv} G\ltimes_\alpha\Hc,\chi)$ and the untwisted groupoid
\[ \Gc:=(G\times G\times\Hc_1\arrows G\times\Hc_0) \]
whose source, range and multiplication are
\begin{enumerate}
\item $s(g,h,\gamma):=(g,h^{-1}s\gamma)\quad r(g,h,\gamma):=(gh^{-1},r\gamma)$
\item $(gh_2^{-1},h_1,\gamma_1)\circ(g,h_2,h_1^{-1}\gamma_2):=(g,h_1h_2,\gamma_1\gamma_2)$
\end{enumerate}
for $g's,h's\in G$ and $\gamma\in\Hc$.  This groupoid has the natural left translation action of $G$ for which it is equivariantly isomorphic to the generalized $G$-bundle
\[ G\rtimes_q(G\ltimes_\alpha\Hc) \]
where $q$ is the quotient homomorphism $q:G\ltimes_\alpha\Hc\to (G\arrows *)$.  The map is given by
\[ \Gc\longrightarrow G\rtimes_q(G\ltimes_\alpha\Hc) \]
\[ (g,h,\gamma)\longmapsto(gh^{-1},h,\gamma). \]
As was mentioned in the construction of generalized bundles, a generalized bundle $G\rtimes_\rho\Gc$ can be described as the induced groupoid of the right $\Gc$-module $P:=G\times\Gc_0$ with the obvious structure.  In the current situation the right $(G\ltimes_\alpha\Hc)$-module is
\[ P:=(G\times\Hc_0\stackrel{b}{\rightarrow}\Hc_0) \]
where the moment map $b$ is just the projection and the right action is
\[ (g,r\gamma)\cdot(h,h\gamma):=(g(q((h,h\gamma),s\gamma)=(gh,s\gamma) \]
For convenience we will write down this groupoid structure.
\[ G\rtimes_q(G\ltimes_\alpha\Hc)\equiv P\rtimes_q(G\ltimes_\alpha\Hc)
\]
with source, range and multiplication
\begin{enumerate}
\item $s(g,h,\gamma):=(gh,h^{-1}s\gamma)\quad r(g,h,\gamma):=(g,r\gamma)$
\item $(g,h_1,\gamma_1)\circ(gh_1,h_2,h_1^{-1}\gamma_2):=(g,h_1h_2,\gamma_1\gamma_2)$
\end{enumerate}
for $g's,h's\in G$ and $\gamma\in\Hc$.

The content of the previous paragraph is that up to a Pontryagin duality, a Takai duality can be expressed purely in terms of groupoids.  Given a $G$-groupoid $(\alpha,\Hc)$, one forms the crossed product $G\ltimes_\alpha\Hc$.  There is a canonical $\widehat{G}$-action on the groupoid algebra, but there is also a canonical right $(G\ltimes_\alpha\Hc)$-module $P$, and the two canonical pieces of data are essentially the same.  Now to express the duality, one passes to the induced groupoid $P\rtimes(G\ltimes_\alpha\Hc)$.  This induced groupoid is itself equipped with a natural $G$-action coming from the left translation of $G$ on $P$, and is $G$-equivariantly Morita equivalent to the original $G$-groupoid $(\alpha,\Hc)$.  Thus the $C^*$-algebra of this induced groupoid takes the place of $\widehat{G}\ltimes_{\hat{\alpha}}C^*(G\ltimes_\alpha\Hc))$ (and is isomorphic to it via Fourier transform).

This duality $(\alpha,\Hc)\rightsquigarrow(P,G\rtimes_\alpha\Hc)$ expresses essentially the same phenomena as Takai duality, but while Takai duality applies only to abelian groups, this formulation applies to arbitrary groups.  Let us write this down formally.

\begin{thm}[\textbf{Nonabelian Takai duality for groupoids}]\label{T:nonabelianTakai}Let $G$ be a locally compact group and $(\alpha,\Hc)$ a $G$-groupoid, then
\begin{enumerate}
\item $G\ltimes_\alpha\Hc$ has a canonical right module $P:=((G\times\Hc_0)\to\Hc_0)$.
\item The induced groupoid $P\rtimes_q(G\ltimes_\alpha\Hc)$ is naturally a generalized principal $G$-bundle.
\item The $G$-groupoids $(\alpha,\Hc)$ and $(\tau,P\rtimes_q(G\ltimes_\alpha\Hc))$ are equivariantly Morita equivalent, where $\tau$ denotes the left principal $G$-bundle action.
\end{enumerate}
\end{thm}
\begin{proof} The first statement has already been explained, and the second is just the fact that
\[ P\rtimes(G\ltimes_\alpha\Hc)\equiv G\rtimes_q(G\ltimes_\alpha\Hc) \] and the latter groupoid is manifestly a $G$-bundle.
For the third statement, note that the inclusion
\[  \Hc\simeq\{1\}\times\{1\}\times\Hc\stackrel{\phi}{\hookrightarrow}(G\rtimes_q(G\ltimes_\alpha\Hc)) \]
is an essential equivalence.

After identifying the equivalence bimodule $P_\phi$ with the space $\{(g,g^{-1},\gamma)\ |\ g\in G,\ \gamma\in\Hc_1\ \}$,
one verifies easily that the following $G$-actions make this an equivariant Morita equivalence:

For  $(h,k,\gamma)\in(G\rtimes_q(G\ltimes_\alpha\Hc))$, $\gamma\in\Hc$, and $(h,h^{-1},\gamma)\in P_\phi$, $g\in G$ acts by
\[ g\cdot(h,k,\gamma):=(gh,k,\gamma);\quad g\cdot(h,h^{-1},\gamma):=(gh,(gh)^{-1},\gamma);\quad g\cdot(\gamma):=g\gamma. \]
Note that $\phi$ itself is not an equivariant map.
\end{proof}

For our intended application to T-duality, it will be necessary to consider a nonabelian Takai duality for twisted groupoids, which we will prove here.  In this context it not possible to make a statement of equivariant Morita equivalence because in general the twisted groupoid does not admit a $G$-action.  Instead of $G$-equivariance, there is a Morita equivalence that is compatible with ``extending'' to a twisted crossed product groupoid.

\begin{thm}[\textbf{Nonabelian Takai duality for twisted groupoids}]\label{T:twistedNonabTakai} Let $G$ be a group and $(\alpha,\Hc)$ a $G$-groupoid.  Suppose we are given a 2-cocycle
$\chi\in Z^2(G\ltimes_\alpha\Hc;\Uone)$.  Define $\sigma:=\chi|_\Hc\in Z^2(\Hc,\Uone)$.  Then $\chi$ can be viewed as a 2-cocycle on $(G\rtimes_q(G\ltimes_\alpha\Hc))$ which is constant in the first $G$ variable, or as a 2-cocycle on $G\ltimes_\tau(G\rtimes_q(G\ltimes_\alpha\Hc))$ which is constant in the first two $G$-variables, and we have:
\begin{enumerate}
\item The Morita equivalence $(G\ltimes(G\rtimes_\alpha\Hc))\simeq\Hc $ of Theorem \eqref{T:nonabelianTakai} extends to a Morita equivalence of $\Uone$-gerbes:
\[ \Uone\rtimes^\chi(G\rtimes_q(G\ltimes_\alpha\Hc))\simeq\Uone\rtimes^\sigma\Hc. \]
\item The Morita equivalence $G\ltimes_\tau(G\rtimes_q(G\ltimes_\alpha\Hc))\simeq G\ltimes_\alpha\Hc $ induced from $G$-equivariance extends to a Morita equivalence of $\Uone$-gerbes:
\[ \Uone\rtimes^\chi(G\ltimes_\tau(G\rtimes_q(G\ltimes_\alpha\Hc)))\simeq\Uone\rtimes^\chi(G\ltimes_\alpha\Hc). \]
\item The first equivalence is a subequivalence of the second.
\end{enumerate}
\end{thm}
\begin{proof}The first statement follows just as in the proof of the Morita equivalence of $G\rtimes_q(G\ltimes_\alpha\Hc)$ and $\Hc$; this time the bimodule is $\Uone\times P_\phi$ (using the notation from the proof of Theorem \eqref{T:nonabelianTakai}).

For the second statement, note that $G\times P_\phi$ is the $G\ltimes_\tau(G\rtimes(G\ltimes_\alpha\Hc))$- $G\ltimes_\alpha\Hc$-bimodule, the bimodule structure being given as in Example \eqref{ex:transformationgroupoid}.  Then $\Uone\times G\times P_\phi$ is the desired Morita bimodule.

For the third statement, note that restricting to $\Uone\times\{1\}\times P_\phi\subset \Uone\times G\times P_\phi$ recovers the Morita equivalence of the first statement as a subequivalence of the second.
\end{proof}

\section{ Nonabelian noncommutative T-duality }\label{S:fullTduality}
Remembering our convention to call a group \emph{nonabelian} when it is not commutative and reserve the word  \emph{noncommutative} for a noncommutative space in the sense of noncommutative geometry, we define the following extensions to T-duality.
\begin{defn} Let $N$ be a closed normal subgroup of a locally compact group $G$.  There is a canonical equivalence between the data contained in:
\begin{itemize}
\item a $G$-equivariant $\Uone$-gerbe on a generalized $G/N$-bundle, and
\item a $\Uone$-gerbe on an $N$-gerbe, with canonical right module.
\end{itemize}
The interpolation between the two objects is described below, and will be called \textbf{nonabelian noncommutative T-duality} whenever the interpolating procedure induces an isomorphism in $K$-theory.
\end{defn}
\begin{rem}
Classical T-duality was a duality between gerbes on generalized principle bundles.  More precisely, for abelian groups $G$, the duality associated to a $\Uone$-extension of a groupoid of the form $G/N\rtimes_\rho\Gc$, a new $\Uone$-extension of a groupoid of the form $\widehat{N}\rtimes_{\bar{\lambda}}\Gc$.  Now we will instead use groupoids of the form $G\rtimes N\rtimes^{\delta\rho}\Gc$ or more generally $G\rtimes G\times_{G/N,\rho}\Gc$, defined in Example \eqref{Ex:bestgroupoid2}.  These are equivariantly Morita equivalent to the old kind.  The reason for this change is that while every gerbe on a classically T-dualizable pair of bundles can be presented by an equivariant 2-cocycle on a groupoid of the form $G/N\rtimes_\rho\Gc$, this is not the case in general.  On the other hand, the following fact will be proved in Section \eqref{S:Brauer}:
\begin{center}
\emph{ Every $G$-equivariant stack theoretic gerbe on a generalized principal $G/N$-bundle admits a presentation as a $\Uone$-extension of a groupoid of the form $G\rtimes G\times_{G/N,\rho}\Gc$.}
\end{center}
This is not necessarily obvious.  In fact without the $G$-equivariance condition, not every gerbe on a generalized $G/N$-bundle admits such a presentation; it is the $G$-equivariance that forces this.
\end{rem}
\noindent\underline{Procedure for nonabelian T-dualization.}  The initial data for nonabelian T-duality will be a $G/N$-bundle with equivariant 2-cocycle:
\begin{equation}\label{E:nonabTdualityData}
(G\rtimes N\rtimes^{\delta\rho}\Gc,(\sigma,\lambda,\beta)\in Z^2_G(G\rtimes N\rtimes^{\delta\rho}\Gc;\Uone)).
\end{equation}
The nonabelian T-dual is obtained in the following two steps:  \newline\newline
\underline{Step 1} Pass to the crossed product groupoid, together with its canonical right module $P$ of Theorem \eqref{T:nonabelianTakai} and the 2-cocycle $\psi$ which is the image of $(\sigma,\lambda,\beta)$ in $Z^2(G\ltimes G\rtimes N\rtimes^{\delta\rho}\Gc;\Uone)$ :
\[ (G\ltimes(G\rtimes N\rtimes^{\delta\rho}\Gc),\psi,P).\]
\underline{Step 2} Pass to the Morita equivalent system:
\[ (N\rtimes^{\delta\rho}\Gc;\iota^*\psi,\iota^*P) \]
where $\iota$ is the essential equivalence (see Proposition \eqref{P:essentialEquivs}):
\[ \iota:\ (N\rtimes^{\delta\rho}\Gc)\hookrightarrow G\ltimes (G\rtimes N\rtimes^{\delta\rho}\Gc) \]
\[ (n,\gamma)\mapsto(n\rho(\gamma),e,n,\gamma). \]
\ \newline\newline
\underline{Some properties of the duality}\newline\newline
(A) If $\rho:\Gc\to G/N$ does not admit a lift to a map to $G$, then the same T-dualization procedure works with $N\rtimes^{\delta\rho}\Gc$ replaced by $G\times_{G/N,\rho}\Gc$ (see Example \eqref{Ex:bestgroupoid1}).
\newline\newline
(B)\quad What we are describing is indeed a duality, in the sense that we can recover the initial system \eqref{E:nonabTdualityData} by inducing the groupoid via its canonical module $P$.  This fact is the content of twisted nonabelian Takai duality for groupoids, Theorem \eqref{T:twistedNonabTakai}.
\newline\newline
(C)\quad There is an isomorphism of twisted $K$-theory
\[ K^\bullet(G\rtimes N\rtimes^{\delta\rho}\Gc;\sigma)\simeq
   K^{\bullet+dimG}(G\ltimes(G\rtimes N\rtimes^{\delta\rho}\Gc),\psi)\simeq K^{\bullet+dimG}(N\rtimes^{\delta\rho}\Gc;\iota^*\psi) \]
whenever $G$ satisfies Connes' Thom isomorphism theorem, in particular whenever $G$ is a (finite dimensional) 1-connected solvable Lie group.
\newline\newline
(D)\quad This construction is Morita invariant in the appropriate sense.  That is, if we choose two representatives for the same generalized principal bundle with equivariant gerbe, the two resulting dualized objects present the same $N$-gerbe with equivariant gerbe.  This is proved in Section \eqref{S:Brauer}.
\newline\newline
(E)\quad The dual can be interpreted as a family of noncommutative groups.  Indeed, suppose that in the above situation $\Gc$ is a \v{C}ech groupoid of a locally finite cover of a space $X$.  Let us look at the fiber of the dual over a point $m\in X$.  This fiber corresponds to $N\rtimes^{\delta\rho}(\Gc|_m)$, where $\Gc|_m$ is the restriction of $\Gc$ to the chosen point.  $\Gc|_m$ is a pair groupoid which is a finite set of points (there is one arrow for each double intersection $U_i\cap U_j$ that contains $m$ and one object for each element of the cover that contains $m$).  Any inclusion of the trivial groupoid $(*\arrows*)\hookrightarrow\Gc|_m$ induces an essential equivalence $\phi:(N\arrows *)\hookrightarrow N\rtimes^{\delta\rho}(\Gc|_m)$ which induces an isomorphism $\phi^*$ in cohomology by Corollary \eqref{C:fundamental2}.   So twisted groupoids $((N\rtimes^{\delta\rho}\Gc)|_m,(\iota^*\psi)|_m)$ and $(N\arrows *,\phi^*(\iota^*\psi|_m))$ are equivalent.  In particular, $C^*(N\rtimes^{\delta\rho}\Gc|_m;\iota^*\psi)$ is Morita equivalent to $C^*(N;\phi^*(\iota^*\psi))$, and the latter is a standard presentation of a noncommutative (and nonabelian, if desired) dual group!  For example if $N\simeq \Zb$ we get noncommutative n-dimensional tori.  So the T-duality applied to $G/N$-bundles produces $N$-gerbes that are fibred in what should be interpreted as noncommutative versions of the dual group $(G/N)^\vee$!

\section{ The equivariant Brauer group }\label{S:Brauer}
In this section we will describe the elements of the $G$-equivariant Brauer group of a principal ``$G/N$-stack''.  First recall that the Brauer group $\Br(X)$ of a space $X$ is the set of isomorphism classes of stable separable continuous trace $C^*$-algebras with spectrum $X$.  The famous Dixmier-Douady classification says that each such algebra is isomorphic to the algebra $\Gamma_0(X;E)$ of sections that vanish at infinity of a bundle $E$ of compact operators.  Since such bundles can be described by transition functions with values in $\Aut K\simeq \puh$, there is an isomorphism $H^1(X;\Aut K)\simeq\Br(X)$.  Since $H^1(X;\Aut K)\simeq H^2(X;\Uone)$, $Br(X)$ can also be taken to classify $\Uone$-gerbes on $X$.

If $X$ is a $G$-space, one can talk about the equivariant Brauer group $\Br_G(X)$, which can be most simply defined as the equivalence classes of $G$-equivariant bundles of compact operators under the equivalence of isomorphism and outer equivalence of actions.  We intend to show that this group also corresponds to $G$-equivariant gerbes on $X$.

Generalizing the case of spaces $X$, one can consider the Brauer group of a presentable topological stack $\Xc$ (that is, a stack $\Xc$ which is equivalent to $\Prin_\Gc$ for some groupoid $\Gc$, as in Section \eqref{S:gerbes}).  So let $\Xc$ be a presentable topological stack.  Then a vector bundle on $\Xc$ is a vector bundle on a groupoid $\Gc$ presenting $\Xc$.  A vector bundle on a groupoid $\Gc$ is a (left) $\Gc$-module $E\to\Gc_0$ which is a vector bundle on $\Gc_0$ and such that $\Gc$ acts linearly (in the sense that the action morphism $s^*E\to r^*E$ is a morphism of vector bundles over the space $\Gc_1$).
An $\Hc$-$\Gc$-Morita equivalence bimodule $P$ makes $P*E\to\Hc_0$ a vector bundle on $\Hc$.  In exactly the same way, one has the notion of a bundle of algebras on a stack, in particular one has bundles of compact operators on a stack.

Finally, the stack with $G/N$-action associated to generalized principal bundle $G/N\rtimes_\rho\Gc$ is what should be called a principal $G/N$-stack.  This data can be presented in at least two ways, for example as a stack over $\Prin_{G/N}$
\[ \Prin_\Gc\longrightarrow \Prin_{G/N}, \]
or as a stack over $\Prin_\Gc$
\[ \Prin_{G/N\rtimes_\rho\Gc}\longrightarrow\Prin_\Gc. \]

Now define the $G$-equivariant Brauer group $\Br_G(\Prin_{G/N\rtimes_\rho\Gc})$ to be the isomorphism classes $G$-equivariant bundles of compact operators on $\Prin_{G/N\rtimes_\rho\Gc}$.
We will show that whenever $\Gc_0$ is contractible to a set of points,
\[ \Br_G(\Prin_{G/N\rtimes_\rho\Gc})\simeq H^1(G\times_{G/N}\Gc;\Aut K). \]
Of course if $\Gc_0$ isn't contractible we can refine $\Gc$ so that it is, thus we will always make that assumption.

So suppose $E$ is a bundle on $\Prin_{G/N\rtimes_\rho\Gc}$.  It may initially be presented as a module over any groupoid $\Hc$ which is equivariantly Morita equivalent to $(G/N\rtimes_\rho\Gc)$, for example if $\Gc$ is a \v{C}ech groupoid one could imagine that $E$ is a bundle on the actual space  $Q\simeq(G/N\times\Gc_0)/(G/N\rtimes_\rho\Gc)$.  We would like $E$ to be presented on the groupoid $G\rtimes(G\rtimes_{G/N,\rho}\Gc)$, (as in Example \eqref{Ex:bestgroupoid2}), so choose a Morita equivalence ($(G\rtimes G\rtimes_{G/N,\rho}\Gc)$-$\Hc$)-bimodule $P$, and replace $E$ by $\tilde E:= P*E$.  Remember that no data is lost here because $P^{\opp}*P*E\simeq E$.

Just for the sake of not having too many $G's$, assume that $\rho$ admits a lift to $G$ so that there is an isomorphism
\[ (G\rtimes G\rtimes_{G/N,\rho}\Gc)\simeq (G\rtimes N\rtimes^{\delta\rho}\Gc) \]
as in Example \eqref{Ex:bestgroupoid2}.  The general case when there is no lift works in the exact same way.

If $E$ is $G$-equivariant then so is $\tilde E$.  Let us suppose this is the case, meaning precisely that $\tilde E$ is a $(G\rtimes N\rtimes^{\delta\rho}\Gc)$-module, and $\tilde E$ has a $G$-action which is equivariant with respect to the translation action of $G$ on $(G\rtimes N\rtimes^{\delta\rho}\Gc)$.

Keep in mind that in particular $\tilde E$ is a bundle over the objects $G\times\Gc_0$ of the groupoid.
Then the restriction of $\tilde E$ to $\{e\}\times\Gc_0\subset G\times\Gc_0$, denoted $E_0$, is trivializable since $\Gc_0$ is contractible.  So assume that $E_0=\Gc_0\times K$.  But then the whole of $\tilde E$ is trivializable since it is a $G$-equivariant bundle over a space with free $G$-action.  For example a trivialization is given by:
\[ G\times E_0\to \tilde E\qquad (g,\xi)\mapsto g\xi.  \]
So assume that $\tilde E= G\times\Gc_0\times K$.

Note that being a $G$-equivariant $(G\rtimes N\rtimes^{\delta\rho}\Gc)$-module is the same as being a $G\ltimes(G\rtimes(N\rtimes^{\delta\rho}\Gc))$-module.  Since $\tilde E$ is trivial as a bundle, it is classified by the action of $G\ltimes(G\rtimes N\rtimes^{\delta\rho}\Gc),$ and this is given by a homomorphism
\[ \pi:G\ltimes(G\rtimes N\rtimes^{\delta\rho}\Gc)\to\Aut K \]
such that the groupoid action is
\[  (G\ltimes(G\rtimes N\rtimes^{\delta\rho}\Gc))\times_{G\times\Gc_0}\tilde E\longrightarrow \tilde E \]
\[ ((g_1,g_2,n,\gamma),(g_1^{-1}g_2n\rho(\gamma),s\gamma,k))\mapsto (g_2,r\gamma,\pi(g_1,g_2,n,\gamma)k) \]
for $(g_1^{-1}g_2n\rho(\gamma),s\gamma,k)\in\tilde E$.

Two such actions, given by $\pi$ and $\pi'$, are outer equivalent if and only if
$\pi$ and $\pi'$ are cohomologous, and this shows that
\[ \Br_G(\Prin_{G/N\rtimes_\rho\Gc})\simeq H^1(G\ltimes G\rtimes N\rtimes^{\delta\rho}\Gc;\Aut K). \]
But the results of Section \eqref{S:cohoFacts} imply that the inclusion
\[ \iota:N\rtimes^{\delta\rho}\Gc\to G\ltimes G\rtimes N\rtimes^{\delta\rho}\Gc\qquad (n,\gamma)\mapsto(n\rho(\gamma),e,n,\gamma) \]
induces a quasi-isomorphism with quasi-inverse the quotient map
\[q:G\ltimes G\rtimes N\rtimes^{\delta\rho}\Gc\to N\rtimes^{\delta\rho}\Gc \]
therefore we have:
\begin{prop} Let $N$ be a closed normal subgroup of a locally compact group $G$, let $\Gc$ be a groupoid with $\Gc_0$ contractible, and let $\rho:\Gc\to G/N$ be a homomorphism.  Then
\[ \Br_G(\Prin_{G/N\rtimes_\rho\Gc})\simeq H^1(G\ltimes G\rtimes N\rtimes^{\delta\rho}\Gc;\Aut K)
     \simeq H^1(N\rtimes^{\delta\rho}\Gc;\Aut K). \]
More generally, $N\rtimes^{\delta\rho}\Gc$ can be replaced by $G\rtimes_{G/N,\rho}\Gc$.
\end{prop}\vspace{-0.1in}
\begin{flushright}$\square$\end{flushright}

The meaning of the isomorphism on the right is that $\pi$ can be taken constant in its two $G$ variables.  In fact, one can construct such a $\pi$ directly; simply note that the whole situation is determined by the module structure of $E_0$, and that it is precisely $\iota(N\rtimes^{\delta\rho}\Gc)$ that preserves this subspace.  An example of this construction is carried out in the proof of Theorem \eqref{T:appendixThm}.

Now let us explain the relationship between gerbes and bundles of compact operators.

If $N$ is a discrete group then there is a connecting homomorphism which is an isomorphism
\[ H^1(N\rtimes^{\delta\rho}\Gc;\Aut K)\simeq H^2(N\rtimes^{\delta\rho}\Gc;\Uone). \]
If $N$ is not discrete then there is the possibility that only a Borel connecting homomorphism can be chosen.  Rather than tread into the territory of Borel cohomology for groupoids, we point out that for any groupoid $\Hc$, the $\Uone$-gerbe associated to $\pi\in Z^1(\Hc;\Aut K)$ is just
\[ U(\hf)\times_{\Aut K,\pi}\Hc,  \]
the groupoid constructed in Example \eqref{Ex:bestgroupoid1}.  It is a $\Uone$-gerbe on $\Hc$.
To make that last fact more clear, note that as a space, the arrows of this groupoid form a possibly nontrivial $\Uone$-bundle over $\Hc_1$ and any global section of the bundle determines an isomorphism
\[ U(\hf)\times_{\Aut K,\pi}\Hc\simeq \Uone\rtimes^{\delta\pi}\Hc, \]
as was pointed out in Example \eqref{Ex:bestgroupoid1}, and the latter object is clearly a $\Uone$ gerbe.

At $C^*$-algebra level there is also a relationship between gerbes and bundles of compact operators.
It is shown in Lemma \eqref{L:appendixLemma2} that a global section of the $\Uone$-bundle $(U(\hf)\times_{\Aut K,\pi}\Hc_1)\to\Hc_1$ induces a Morita equivalence
\[ C^*(\Hc,\delta\pi)\Morita\Gamma(\Hc;E(\pi)) \]
where $E(\pi)$ is the trivial bundle $\Hc_0\times K$ with $\Hc$-module structure given by $\pi$.
Independent of the existence of a global section, there is a Morita equivalence
\[ \Gamma(\Hc;Fund(U(\hf)\times_{\Aut K,\pi}\Hc))\Morita\Gamma(\Hc;E(\pi)) \]
where $Fund(U(\hf)\times_{\Aut K, \pi}\Hc)$ denotes the associated line bundle to the $\Uone$-bundle (viewed as a bundle of (rank 1) $C^*$-algebras) $(U(\hf)\times_{\Aut K,\pi}\Hc_1)\To\Hc_1$.

So in this section we saw that nonabelian noncommutative T-duality as presented in Section \eqref{S:fullTduality} can be used to describe a dual to a $G$-equivariant gerbe presented on \emph{any} groupoid $\Hc$ such that $\Hc$ describes a principal ``$G/N$-stack''; this being because any such gerbe could also be presented on $(G\rtimes N\rtimes^{\delta\rho}\Gc)$, as it was in Section \eqref{S:fullTduality}.  (Except for the situation in which $\rho$ and possibly $\pi$ do not admit lifts to $G$ or $U(\hf)$ respectively, but we also saw how to modify the setup in these situations).

Finally, as promised:
\begin{cor} Nonabelian T-duality is Morita invariant. \end{cor}
\begin{proof} This is easy because we may assume that any gerbe on a generalized principal bundle is presented on a groupoid of the form $G\rtimes N\rtimes^{\delta\rho}\Gc$ and that the gerbe is defined via a cocycle which is constant in $G$.  But then if two such presentations are given, corresponding to $\pi\in Z^1(N\rtimes^{\delta\rho}\Gc;\Aut K)$ and $\pi'\in Z^1(N\rtimes^{\delta\rho'}\Gc';\Aut K)$ it is clear that the gerbes $U(\hf)\times_{\Aut K,\pi}G\rtimes N\rtimes^{\delta\rho}\Gc$ and $U(\hf)\times_{\Aut K,\pi'}G\rtimes N\rtimes^{\delta\rho'}\Gc'$ are $G$-equivariantly Morita equivalent if and only if $U(\hf)\times_{\Aut K,\pi} N\rtimes^{\delta\rho}\Gc$ and $U(\hf)\times_{\Aut K,\pi'} N\rtimes^{\delta\rho'}\Gc'$ are Morita equivalent.
\end{proof}

\section{Conclusion}
A natural direction for future study here is to consider the case when both sides of the duality are fibred in noncommutative groups (in the sense of noncommutative geometry).  Interestingly, a completely new phenomenon arises in this context: the gerbe data, the 2-cocycles that is, become noncompactly supported distributions and it is necessary to multiply them.  We have manufactured examples in which this can be done, that is when the singular support of the distributions do not intersect, but our present methods do not provide a general method for describing a T-dual pair with both sides families of noncommutative groups.

Another direction to look is the case of groupoids for which the $G/N$-action is only free on a dense set.  This corresponds to singularities in the fibers of a bundle of groups.  The groupoid approach to T-duality seems well suited for this.  On the other hand for some other types of singularities in fibers, notably singularities which destroy the possibility of a global $G/N$-action, the groupoid approach will not apply at all.  It will be interesting to see if this problem can be fixed.

A third direction these methods can take is to consider complex structures on the groupoids and make the connection between topological T-duality and the T-duality of complex geometry (as in \cite{DP}).  We have initiated this project in \cite{BD}.

Lastly, it will of course be very nice to find some physically motivated examples of nonabelian T-duality.

\appendix\section{Connection with the Mathai-Rosenberg approach}
The goal of this section is to describe the connection between our approach and the Mathai-Rosenberg approach to T-duality.

Let us begin with a summary of the approach of Mathai and Rosenberg \cite{MR}.  One begins with the data of a principal torus bundle $P\to X$ over a space $X$ and a cohomology class $H\in H^3(P;\Zb)$ called the $H$-flux.  The procedure for T-dualizing is as follows.
\begin{enumerate}
\item Pass from the data $(P,H)$ to a $C^*$-algebra $A(P,H)$.  To do this one traces $H$ through the isomorphisms
\begin{equation}\label{E:cohoIsoms}
H^3(P;\Zb)\simeq \check{H}^2(P;\Uone)\simeq\check{H}^1(P;\Aut(K))
\end{equation}
(here $K=K(\hf)$ is the algebra of compact operators on a fixed separable Hilbert space $\hf$) to get an $\Aut(K)$-valued \v{C}ech 1-cocycle.  This cocycle gives transition functions for a bundle of compact operators over $P$, and $A(P,H)$ is the $C^*$-algebra of continuous sections of this bundle which vanish at infinity.  According to the Dixmier-Douady classification, one can recover the torus bundle $P$ and the $H$-flux from the $A(P,H)$.
\item Next, writing the torus as a vector space modulo full rank lattice, $T=V/\Lambda$, one tries to lift the action of $V$ on $P$ to an action of $V$ on $A(P,H)$.  Assuming one exists, choose an action $\alpha:V\to \Aut(A(P,H))$ lifting the principal bundle action.  If no action exists, the data is not T-dualizable.
\item Now the T-dual of $A=A(P,H)$ is simply the crossed product algebra, $V\ltimes_\alpha A$, (or perhaps the T-dual is the spectrum and Dixmier-Douady invariant $(P^\vee,H^\vee)$ of this algebra, if $V\ltimes_\alpha A$ is of continuous trace) and the problem of producing a T-dual object is reduced to understanding this crossed product algebra.
\item There are two scenarios for describing the crossed product, depending on whether a certain obstruction class, called the Mackey obstruction $M(\alpha)$, vanishes.  When $M(\alpha)=0$ the crossed product algebra is isomorphic to one of the form $A(P^\vee,H^\vee)$ for $P^\vee\to X$ a principal dual-torus bundle and $H^\vee\in H^3(P^\vee,\Zb)$.  The transition functions for $P^\vee$ are obtained from the so-called Phillips-Raeburn obstruction of the action and there exist explicit formulas for $H^\vee$ in terms of the data $(P,H,\alpha)$.  This understanding of the crossed product is a result of work by Mackey, Packer, Phillips, Raeburn, Rosenberg, Wassermann, Williams and others (in the subject of crossed products of continuous trace algebras) which is referenced in \cite{MR}.

When $M(\alpha)\neq 0$, the crossed product algebra was shown in \cite{MR} to be a continuous field of stable noncommutative (dual) tori over $X$.
\end{enumerate}
We claim that the Mathai-Rosenberg setup corresponds to our approach applied to \v{C}ech groupoids and the groups $(G,N)=(V,\Lambda)$.  More precisely, we have the following theorem.
\begin{thm}\label{T:appendixThm} Let $Q\to X$ be a principal torus bundle trivialized over a good cover of $X$, let $\Gc$ denote the \v{C}ech groupoid for this cover and let $\rho:\Gc\to V/\Lambda$ be transition functions presenting $Q$.  Then \begin{enumerate}
\item For any $H\in H^3(Q;\Zb)$ such that $A:=A(Q;H)$ admits a $V$-action, there is a Morita equivalence
\[ A\stackrel{Morita}{\sim}C^*(V\rtimes \Lambda\rtimes^{\delta\rho}\Gc;\sigma), \]
for some $\sigma\in Z^2(V\rtimes\Lambda\rtimes^{\delta\rho}\Gc;\Uone)$ that is constant in $V$.  If $V$ acts by translation on $C^*(V\rtimes \Lambda\rtimes^{\delta\rho}\Gc;\sigma)$ then the equivalence is $V$-equivariant.
\item $[\sigma]$ is the image of $H$ under the composite map
\[ H^3(Q;\Zb)\mTo{\sim}H^2(Q;\Uone)\To H^2(V\rtimes\Lambda\rtimes^{\delta\rho}\Gc;\Uone).\]
\item Let $\sigma^\vee:=\sigma|_{\Lambda\rtimes^{\delta\rho}\Gc}\in Z^2(\Lambda\rtimes^{\delta\rho}\Gc;\Uone)$. Then for the chosen action of $V$, there is a $\widehat{V}$-equivariant Morita equivalence:
\[ V\ltimes A\stackrel{Morita}{\sim}V\ltimes C^*(V\rtimes\Lambda\rtimes^{\delta\rho}\Gc;\sigma)
             \stackrel{Morita}{\sim}C^*(\Lambda\rtimes^{\delta\rho}\Gc;\sigma^\vee) \]
where $\widehat{V}$ acts by the canonical dual action on the left two algebras and on the rightmost algebra by $\phi\cdot a(\lambda,\gamma):=\langle\phi,\lambda\rho(\gamma)\rangle a(\lambda,\gamma)$ for $\phi\in\widehat{V}$ and $(\lambda,\gamma)\in\Lambda\rtimes^{\delta\rho}\Gc.$
\end{enumerate}
\end{thm}
The theorem will follow from the next two lemmas concerning bundles of $C^*$-algebras.
\begin{defn} Let $\Gc$ be a groupoid.  A (left)\textbf{ $\Gc$-$C^*$-algebra $\Ac$} is a bundle of $C^*$-algebras $\Ac\to\Gc_0$ which is a (left) $\Gc$-module, such that $\Gc$ acts by $C^*$-algebra isomorphisms.  The \textbf{groupoid algebra of sections of $\Ac$}, written $\Gamma^*(\Gc;r^*\Ac)$, is the $C^*$-completion of $\Gamma_c(\Gc;r^*\Ac)$ (= the compactly supported sections of the pullback bundle $r^*\Ac$) with multiplication and involution given by
\[ ab(g):=\int_{g_1g_2=g}a(g_1)g_1\cdot(b(g_2))\qquad a^*(g):=g\cdot(a(g^{-1})^*) \]
for $g's\in\Gc$ and $a,b\in\Gamma_c(\Gc;r^*\Ac)$, and where the last ``$\ ^*\ $'' on the right is the $C^*$-algebra involution in the fiber.
\end{defn}
\begin{rem} This definition is a synonym for the groupoid crossed product algebra $\Gc\ltimes\Gamma_0(\Gc_0;\Ac)$, as defined in \cite{Ren2}.  \end{rem}
\begin{lem}\label{L:appendixLemma1} Let $\Gc$ and $\Hc$ be groupoids and $(\Gc_0\negthickspace \stackrel{\ell}{\leftarrow}P\mto{\rho}\Hc_0)$ a $(\Gc$-$\Hc)$-Morita equivalence bimodule.  Then for any $\Hc$-$C^*$-algebra $\Ac\mto{\pi}\Hc_0$, there is a Morita equivalence
\[ \Gamma^*(\Gc;r^*(P*\Ac))\stackrel{Morita}{\sim}\Gamma^*(\Hc;r^*\Ac)\]
where, as in Section \eqref{S:groupoidmoritaeq},
\[ P*\Ac:=(P\times_{\rho,\Hc_0,\pi}\Ac)/\Hc)=((P\times_{\rho,\Hc_0,\pi}\Ac)/(p,a)\sim(ph^{-1},h\cdot a)).\]
\end{lem}
\begin{proof}  We will construct a Morita equivalence bimodule for essential equivalences.  If this case is true then the lemma is true because any Morita equivalence factors into essential equivalences, and if $\Gc\stackrel{\phi_1}{\leftarrow}\Kc\mto{\phi_2}\Hc$ is such a factorization then, setting $P=P_{\phi_1}^{op}*P_{\phi_2}$, we will have
\[ \Gamma^*(\Hc;r^*\Ac)\Morita\Gamma^*(\Kc;r^*(P_{\phi_2}^*\Ac))\stackrel{iso}{\sim}
    \Gamma^*(\Kc;r^*P_{\phi_1}*(P*\Ac))\Morita\Gamma^*(\Gc;r^*(P*\Ac)).  \]

So in the notation of the statement of the lemma, assume
\[ P=P_\phi:=\Gc_0\times_{\phi,\Hc_0,r}\Hc_1 \]
for $\phi:\Gc\to\Hc$ an essential equivalence.  Then $\ell$ is the projection $P_\phi\to\Gc_0$ and $\rho$ is the projection $\pi^\Hc:P_\phi\to\Hc_1$ followed by the source map $s:\Hc_1\to\Hc_0$.
We use the isomorphism $P_\phi*\Ac\simeq\phi^*\Ac:=\Gc_0\times_{\phi,\Hc_0,\pi}\Ac$.

Now we will construct a Morita equivalence bimodule.  Set
\[ A_c:=\Gamma_c(\Gc;r^*(\phi*A)) \text{ and } B_c:=\Gamma_c(\Hc;r^*\Ac).\]
The following structure defines an $A_c$-$B_c$-pre-Morita equivalence bimodule structure on $X_c:=\Gamma_c(P;\ell^*\phi^*\Ac)$, which after completion produces the desired Morita equivalence.  Fix the notation:
\[ g's\in \Gc,\ h's\in\Hc,\ p's\in P,\ a's\in A_c,\ x's\in X_c,\ \text{ and }\ b's\in B_c,\]
and as usual integration is with respect to the fixed Haar system.

The left pre-Hilbert module structure is given by the following data:
\begin{itemize}
\item Action: $ax(p):=\int_g a(g)\phi(g)\cdot(x(g^{-1}p)).$
\item Inner product: $_A\langle x_1,x_2\rangle(g):=\int_px_1(gp)\phi(g)\cdot x_2(p)^*$.
\end{itemize}
The right pre-Hilbert module structure is given by the following data:
\begin{itemize}
\item Action: $x b(p):=\int_h h\cdot(x(ph))\pi^\Hc(ph)\cdot b(h^{-1}).$
\item Inner product: $\langle x_1,x_2\rangle_B(h):=\int_p \pi^\Hc(p)^{-1}\cdot(x_1(p)^*x_2(ph))$.
\end{itemize}
Verification that this determines a Morita equivalence bimodule is routine.
\end{proof}
\begin{lem}\label{L:appendixLemma2} Let $\Gc$ be a groupoid, $\sigma:\Gc_2\to\Uone$ a 2-cocycle, and $T:\Gc_1\to U(\hf)$ a continuous map satisfying $\sigma(\gamma_1,\gamma_2)=T(\gamma_1)T(\gamma_2)T(\gamma_1\gamma_2)^{-1}=:\delta T(\gamma_1,\gamma_2).$
Let $\Ac(\ad T)$ denote the $\Gc$-$C^*$-algebra $(\Gc_0\times K)\to\Gc_0$ with $\Gc$-action given by:
\[ s^*\Ac(\ad T)\To r^*\Ac(\ad T)\qquad g\cdot(sg,k):=(rg,\ad T(g)k).\]
Then
\begin{enumerate}
\item There is a Morita equivalence
\[ \Gamma^*(\Gc;r^*\Ac(\ad T))\stackrel{Morita}{\sim}C^*(\Gc;\sigma). \]
\item When $\Gc=\check{\Gc}$ is a \v{C}ech groupoid on a good cover of a space $X$, every $\check{\Gc}$-$C^*$-algebra of compact operators is isomorphic to $\Ac(\ad T)$ for some $T$ and
\[\Gamma_0(X;E(\ad T))\stackrel{Morita}{\sim}C^*(\check{\Gc};\sigma) \]
where $E(\ad T):=\Gc_0\times K/(s\gamma,k)\sim(r\gamma,ad T(\gamma)k)$ is the bundle of compact operators whose transition functions are $\ad T:\Gc\to\Aut K$.
\end{enumerate}
\end{lem}
\begin{proof}
The Morita equivalence in the second statement is proved as follows.  First note that $E(\ad T)=(P_\phi^{op})*\Ac(\ad T)$ where $P_\phi$ is the bimodule of the essential equivalence $\Gc\mTo{\phi}\Gc_0/\Gc\equiv X$, so an application of Lemma \eqref{L:appendixLemma1} implies that
\[ \Gamma_0(X;E(\ad T))\stackrel{Morita}{\sim}\Gamma^*(\check{\Gc};r^*\Ac(\ad T)). \]
Now apply the Morita equivalence of the first statement to finish.  The other part of statement (2), that all bundles of compact operators on $X$ are of the form $E(\ad T)$ for some $T$, follows because the connecting homomorphism in nonabelian \v{C}ech cohomology, $H^1(X;Aut K)\to H^2(X;\Uone)$, is an isomorphism.

It remains to exhibit the Morita equivalence of statement (1).
Set
\[ A_c:=\Gamma_c(\Gc;r^*\Ac(\ad T)) \text{ and } B_c:=C^*(\Gc;\sigma).\]
We claim that the space $X_c:=C_c(\Gc_1;\hf)$ of compactly supported $\hf$-valued maps admits a pre-Morita equivalence $\Gamma^*(\Gc;r^*\Ac(\ad T))$-$C^*(\Gc;\sigma)$-bimodule structure.  Set the notational conventions:
\[ g's\in \Gc,\ a's\in A_c,\ x's\in X_c,\ \text{ and }\ b's\in B_c.\]
The unadorned bracket $\langle\ ,\ \rangle$ denotes the $\Cb$-valued inner product on $\hf$ and is taken to be conjugate-linear in the first variable and linear in the second. The bra-ket notation will be used for the $K$-valued inner product on $\hf$, so for $v's\in\hf$, $|v_1\rangle\langle v_2|$ is the compact operator defined by $|v_1\rangle\langle v_2|(v):=\langle v_2,v\rangle v_1.$

The left pre-Hilbert module structure on $X_c$ is given by the following data:
\begin{itemize}
\item Action: $ax(g):=\int_{g_1g_2=g} \sigma(g_1,g_2)a(g_1)T(g_1)x(g_2)$.
\item Inner product: $_A\langle x_1,x_2\rangle(g):=\int_{g_1g_2=g}\sigma(g_1,g_2)|x_1(g_1)\rangle\langle T(g)x_2(g_2^{-1})|$.
\end{itemize}
The right pre-Hilbert module structure on $X_c$ is given by the following data:
\begin{itemize}
\item Action: $x b(g):=\int_{g_1g_2=g}\sigma(g_1,g_2) x(g_1)b(g_2).$
\item Inner product: $\langle x_1,x_2\rangle_B(g):=\int_{g_1g_2=g}\langle x_1(g_1^{-1}),x_2(g_2)\rangle\sigma(g_1,g_2)$.
\end{itemize}
Verification that this structure gives a Morita equivalence is routine.
\end{proof}

Now let us proceed to the proof of the theorem.
\begin{proof}( Theorem )  By the Dixmier-Douady classification we know that $A$ is isomorphic to the algebra of continuous sections of a bundle $E\to Q$ of compact operators, so assume $A=\Gamma_0(Q;E)$.  A $V$-action on $A$ comes from an action by automorphisms of the bundle of algebras.
Now, noting that $Q$ is isomorphic to the quotient $(V/\Lambda\rtimes_\rho\Gc)_0/(V/\Lambda\rtimes_\rho\Gc)_1$, pull back $E$ to a bundle $\tilde{E}$ over $V/\Lambda\times\Gc_0$.  Then $\tilde{E}$ is a module for the groupoid $V\ltimes V/\Lambda\rtimes_\rho\Gc$.  Denote by $E_0$ the restriction of $\tilde{E}$ to $\{e\Lambda\}\times\Gc_0\subset V/\Lambda\times\Gc_0$.  Then $E_0$ is trivializable since $\Gc_0$ is contractible, so we assume $E_0=\Gc_0\times K$, and write $(s\gamma,k)$ for a point in $E_0\subset\tilde{E}$, and $[s\gamma,k]$ for its image under the quotient $q:\tilde{E}\to E$.  The $V\times V/\Lambda\rtimes^{\delta\rho}\Gc$-module structure on $\tilde{E}$ ``restricts'' to a $\Lambda\rtimes^{\delta\rho}\Gc$-module structure on $E_0$, via the inclusion
\[  \iota:\Lambda\rtimes^{\delta\rho}\Gc\hookrightarrow V\ltimes V/\Lambda\rtimes_\rho\Gc\qquad
   (\lambda,\gamma)\mapsto(\lambda\rho(\gamma),e\Lambda,\gamma).   \]
The action on $E_0$ can be written as
\[(\lambda,\gamma)\cdot(s\gamma,k):=(r\gamma,\pi(\lambda,\gamma)k)\]
where $\pi$ is a homomorphism $\Lambda\rtimes^{\delta\rho}\Gc\to \Aut K$.  (Note that there exists a lift of $\rho$ to $V$ since $\Gc_0$ is contractible and $V\to V/\Lambda$ is a covering space, so $\delta\rho$ makes sense.)

Then $q$ followed by the $V$ action determines a map
\[ V\times E_0\to E\qquad (v,(s\gamma,k))\mapsto (v,[s\gamma,k])\mapsto v\cdot[s\gamma,k]. \]
This map factors through the quotient
\[ V\times E_0\longrightarrow (V\times E_0)/(\Lambda\rtimes^{\delta\rho}\Gc):=(V\times E_0)/(v,(s\gamma,k))\sim(v-\lambda-\rho(\gamma),(r\gamma,\pi(\lambda,\gamma)k)) \]
and induces an isomorphism of bundles
\[ (V\times E_0)/(\Lambda\rtimes^{\delta\rho}\Gc)\ \stackrel{\sim}{\To}\ E \]
which is $V$-equivariant when the bundle on the left is equipped with the natural translation action.

So $A(Q;E)$ is equivariantly isomorphic to $\Gamma(Q;(V\times E_0)/(V\rtimes\Lambda\rtimes^{\delta\rho}\Gc))$.

Let $\sigma:=\delta\pi:(\Lambda\rtimes^{\delta\rho}\Gc)_2\to\Uone$ be an image of $\pi$ under the composition of the connecting homomorphism (which is an isomorphism due to the contractibility of $U(\hf)$) and the pullback via the quotient map $V\rtimes\Lambda\rtimes^{\delta\rho}\Gc\to\Lambda\rtimes^{\delta\rho}\Gc$
\[ H^1(\Lambda\rtimes^{\delta\rho}\Gc;\Aut K)\to H^2(\Lambda\rtimes^{\delta\rho}\Gc;\Uone)\to
      H^2(V\rtimes\Lambda\rtimes^{\delta\rho}\Gc;\Uone).  \]
In other words, we have chosen a continuous map $T:(\Lambda\rtimes^{\delta\rho}\Gc)\to U(\hf)$ such that $\ad T=\pi$ and $\delta T=\sigma$, and $E\simeq E(\ad T)$.  Now we know that $A(Q;H)\simeq\Gamma_0(Q;E)=\Gamma_0(Q;\ad T))$, and according to Lemma \eqref{L:appendixLemma2} there is a Morita equivalence:
\[ C^*(V\rtimes\Lambda\rtimes^{\delta\rho}\Gc;\sigma)\Morita\Gamma_0(Q;E(\ad T)) \]
which is easily seen to be equivariant since $T$ does not depend on $V$.  So statement (1) is proved, and statement (2) is obvious from the construction since $\Gamma_0(Q,E(\ad T))\simeq A(Q,H)$ when $H$ is the image of $[\sigma]=[\delta T]$.

Statement (3) now follows as well.  Indeed, since an equivariant Morita equivalence induces an equivalence of the associated crossed product algebras, we have
\[ V\ltimes C^*(V\rtimes\Lambda\rtimes^{\delta\rho}\Gc;\sigma)\stackrel{Morita}{\sim}V\ltimes A(Q;H) \]
and the algebra on the left, being identical to $C^*(V\ltimes V\rtimes\Lambda\rtimes^{\delta\rho}\Gc;\sigma)$, is equivariantly Morita equivalent to $C^*(\Lambda\rtimes^{\delta\rho}\Gc;\sigma)$ by Proposition \eqref{P:essentialEquivs}.  This completes the proof.
\end{proof}

So that is the correspondence:  Mathai-Rosenberg do $A(Q;H)\leftrightarrow V\ltimes A(Q;H)$, whereas we do $(V\rtimes\Lambda\rtimes^{\delta\rho}\Gc;\sigma)\leftrightarrow (\Lambda\rtimes^{\delta\rho}\Gc;\sigma^\vee)$.  In Sections \eqref{S:classicalT-duality} and \eqref{S:fullTduality} the presentation is slightly different.  The difference is that in this appendix we have assumed that $\sigma$ is already given as a 2-cocyle on  $\Hc:=(V\rtimes\Lambda\rtimes^{\delta\rho}\Gc)$ which is constant in the $V$-direction, whereas in Section \eqref{S:fullTduality} we begin with an arbitrary $\sigma$ on $\Hc$ that has been extended to an equivariant 2-cocyle $(\sigma,\lambda,\beta)$.  The two setups are essentially the same because according to Section \eqref{S:cohoFacts}, the existence of the lift $(\sigma,\lambda,\beta)$ ensures that $\sigma$ is cohomologous to a 2-cocycle which is constant in the $V$ direction.  In Section \eqref{S:classicalT-duality} there is the further difference that $\sigma$ is presented on $V/\Lambda\rtimes_\rho\Gc$ rather than $\Hc$, but we can easily pull it back to a cocycle on $\Hc$.
The slightly messier presentation in Sections \eqref{S:classicalT-duality} and \eqref{S:fullTduality} appeals to the notion that the initial data is a gerbe on a principal bundle (with any groupoid presentation), and that we have found an action of $V$ (that is a lift to an equivariant cocycle) for which the gerbe is equivariant.

The Mackey obstruction in our setup is simply $\beta:=\sigma|_{\Lambda\rtimes^{\delta\rho}\Gc_0}$.  The methods developed in Section \eqref{S:cohoFacts} make it clear that since $\Gc$ is a \v{C}ech groupoid, the restriction to a point in the base space, $\Gc\rightsquigarrow \Gc|_m$ identifies $\beta$ with a 2-cocycle on $\Lambda$.  When $\beta$ is a coboundary, we may assume that $\sigma$ only depends on one copy of $\Lambda$, and the Pontryagin duality methods apply.  Indeed, we may always assume that $\sigma$ is in the image of equivariant cohomology, $H_V^2(V\rtimes\Lambda\rtimes^{\delta\rho}\Gc;\Uone)$, and then $\beta$ really corresponds to the component that obstructs classical dualization and Pontryagin dualization described in Section \eqref{S:classicalT-duality}.


\end{document}